\documentclass[review, 3p, sort&compress, times]{elsarticle}

\usepackage[utf8]{inputenc}
\usepackage[T1]{fontenc}
\usepackage{lmodern}
\usepackage[english]{babel}

\usepackage{lineno} 

\usepackage{tumcolors}
\usepackage{tummacros}
\colorlet{blue}{tumblue}

\usepackage{booktabs} 
\usepackage{graphicx} 
\usepackage{pgfplots} 
\usepackage{amssymb, amsmath, mathtools} %
\pgfplotsset{compat=1.18}
\usepackage{hyperref} 
\usepackage{cleveref} 

\usepackage{etoolbox} 
\newcommand*\linenomathpatch[1]{%
  \cspreto{#1}{\linenomath}%
  \cspreto{#1*}{\linenomath}%
  \csappto{end#1}{\endlinenomath}%
  \csappto{end#1*}{\endlinenomath}%
}
\linenomathpatch{equation}
\linenomathpatch{gather}
\linenomathpatch{multline}
\linenomathpatch{align}
\linenomathpatch{alignat}
\linenomathpatch{flalign}

\usepackage{tabularx}

\usepackage{subcaption}
\usepackage{caption}
\usepackage[obeyFinal, backgroundcolor=white, linecolor=tumgray, textcolor=red, textsize=footnotesize]{todonotes}

\providecommand{\data}[1]{{\par\small\noindent\textbf{\textit{Data availability:}} #1}}
\providecommand{\funding}[1]{{\par\small\noindent\textbf{\textit{Funding:}} #1}}

\newcommand\titlename{Numerical simulation of dilute polymeric fluids with memory effects in the turbulent flow regime}

\hypersetup{
    pdftitle={\titlename},
	pdfauthor={Jonas Beddrich, Stephan B. Lunowa, Barbara  Wohlmuth},
	linkcolor=\linkcolor,
	urlcolor=\urlcolor,
	citecolor=\citecolor,
	menucolor=\menucolor
}

\journal{J. Comput. Phys.}

\begin{document}

\begin{frontmatter}

\title{\titlename}

\author[TUM]{Jonas Beddrich\corref{cor}}
\cortext[cor]{Corresponding author}
\ead{beddrich@ma.tum.de}
\ead[url]{orcid.org/0000-0001-9025-2292}

\author[TUM]{Stephan B. Lunowa}
\ead{stephan.lunowa@tum.de}
\ead[url]{orcid.org/0000-0002-5214-7245}

\author[TUM]{Barbara Wohlmuth}
\ead{wohlmuth@cit.tum.de}
\ead[url]{orcid.org/0000-0001-6908-6015}

\affiliation[TUM]{
    organization={Technical University of Munich, School of Computation, Information and Technology, Department~of~Mathematics},
    addressline={Boltzmannstraße 3},
    postcode={D-85748},
    city={Garching b. München},
    country={Germany}
}

\begin{abstract}
    We address the numerical challenge of solving the Hookean-type time-fractional Navier--Stokes--Fokker--Planck equation, a history-dependent system of PDEs defined on the Cartesian product of two $d$-dimensional spaces in the turbulent regime. Due to its high dimensionality, the non-locality with respect to time, and the resolution required to resolve turbulent flow, this problem is highly demanding. 

To overcome these challenges, we employ the Hermite spectral method for the configuration space of the Fokker--Planck equation, reducing the problem to a purely macroscopic model. 
Considering scenarios for available analytical solutions, we prove the existence of an optimal choice of the Hermite scaling parameter. With this choice, the macroscopic system is equivalent to solving the coupled micro-macro system. We apply second-order time integration and extrapolation of the coupling terms, achieving, for the first time, convergence rates for the fully coupled time-fractional system independent of the order of the time-fractional derivative. 

Our efficient implementation of the numerical scheme allows turbulent simulations of dilute polymeric fluids with memory effects in two and three dimensions. Numerical simulations show that memory effects weaken the drag-reducing effect of added polymer molecules in the turbulent flow regime. 
\end{abstract}


\end{frontmatter}


\section{Introduction}
\label{sec:intro}
The drag-reducing effect of additive polymer molecules was observed for the first time by Toms in 1948 \cite{Toms1948Some}. 
Today, bio-based and synthetic drag-reducing agents are widely used in various industrial applications, including drilling fluids \cite{lu1988carboxymethyl,thombare2016guar}, sprinkler irrigation systems \cite{singh2000biodegradable}, and oil transport, most notably, in the Alaska pipeline \cite{burger1982flow}. 
While the rheological mechanisms of drag-reducing agents in blood flow are not yet fully understood \cite{li2022effect}, additive polymer molecules have been found to reduce occlusion levels in mice with sickle cell disease \cite{crompton2020drag} and to improve microvascular cerebral blood flow in rats with traumatic brain injury \cite{bragin2021addition}.
Most fascinatingly, even minute quantities ($200$--$400\,$ppm) of bio-based polymer additives in water can lead to drag reductions of up to 80\% \cite{han2017role}.

Small amounts of polymer molecules added to an incompressible solvent fluid are commonly modeled as dilute polymeric fluids.
Due to their marginal number, the mass of the polymer molecules and polymer-polymer interactions, such as entanglement, are neglected in the description of the mixture, and only polymer-solvent interactions are considered.
We denote the time with $t \in (0, T)$, the position in the macroscopic flow domain $\CoordinateSpace$ with $\Coordinate \in \CoordinateSpace$. 
The velocity $\Velocity(\Time, \Coordinate)$ and pressure $\Pressure(\Time, \Coordinate)$ of the solvent fluid are governed by the incompressible Navier--Stokes (NS) equations, whereby the influence of the polymer molecules enters as an extra-stress tensor $\ExtraStress(\Time, \Coordinate)$.
In its dimensionless form, the system reads as  
\begin{alignat}{3}
    \label{eq:Navier_Stokes}
    \PDiff{\Time} \Velocity
    + \left( \Velocity \cdot \Grad_{\Coordinate} \right) \Velocity 
    + \Grad_{\Coordinate} \Pressure 
    - \frac{\beta}{\Reynolds} \Laplace_{\Coordinate} \Velocity 
    & = \Grad_{\Coordinate} \cdot \ExtraStress \quad 
    && \text{ in } (0,T) \times \Omega, \\ 
    \Grad_{\Coordinate} \cdot \Velocity & = 0 \quad 
    && \text{ in } (0,T) \times \Omega, 
\end{alignat}
where $\Reynolds$ denotes the Reynolds number and $\beta = \eta_s / (\eta_s + \eta_p)$ the viscosity ratio of the solvent viscosity $\eta_s$ and the zero-shear-rate polymeric viscosity $\eta_p$.
In the case of a pure solvent fluid $\beta = 1$, $\ExtraStress = \Vec{0}$, and we recover the standard incompressible NS equations.

On the microscopic scale, we model a single polymer molecule as a so-called dumbbell, a pair of mass-less beads connected by a spring. 
We identify the center-of-mass of each molecule with the spatial coordinate $\Coordinate$ and denote the configuration vector, the end-to-end vector of the polymer molecule, as $\Configuration \in \ConfigurationSpace$, whereby $\ConfigurationSpace$ is the so-called configuration space.
Considering a Hookean elasticity model, the spring force $\SpringForce$ increases linearly with the extension of the polymer molecule
\begin{equation}
    \label{eq:SpringForceHookean}
    \SpringForce = \SpringConstant \Configuration, \quad \Configuration \in \ConfigurationSpace = \R^d, 
\end{equation}
with $H=1$ after non-dimensionalization.
Despite allowing for unphysical phenomena such as infinitely extended molecules, this modeling approach accurately portrays the behavior of long-chain polymer molecules \cite{herrchen1997detailed}, and it is widely used in practice for its simplicity \cite{barrett2018existence}.
For short polymer chains, commonly, the finitely extensible nonlinear elastic (FENE) model \cite{byron1987dynamics} is used, which only allows for configuration vectors within an open ball, and the spring force tends to infinity as the configuration vector approaches the boundary of the configuration space 
\begin{equation}
    \label{eq:SpringForceFene}
    \SpringForce = \frac{\SpringConstant \Configuration}{1 - |\Configuration|^2 / q_{\max}^2}, \quad 
    \Configuration \in \ConfigurationSpace = B_{q_{\max}} (0),
\end{equation}
where $q_{\max} > 0$ denotes the maximal distance between the endpoints of the polymer molecule. As $q_{\max}$ tends to infinity, we recover the Hookean spring force \cite{warner1972kinetic}. 

On the microscopic scale, the Fokker--Planck (FP) equation describes the dynamics of the polymer molecules in terms of the probability density function $\FPpdf(\Time, \Coordinate, \Configuration)$, expressing the probability that at time $\Time$, there is a polymer molecule with center-of-mass $\Coordinate$ and configuration vector $\Configuration$.
It reads as 
\begin{multline}
    \label{eq:FokkerPlanck}
    \PDiff{\Time} \FPpdf 
    = - \left( \Velocity \cdot \Grad_{\Coordinate} \right) \FPpdf 
    + \comdiff \Laplace_{\Coordinate} \FPpdf
    - \Grad_{\Configuration} \cdot \big((\Grad_{\Coordinate} \Velocity) \Configuration \FPpdf \big)
    + \frac{1}{2 \Deborah} \Grad_{\Configuration} \cdot \big(\SpringForce \FPpdf + \Grad_{\Configuration}\FPpdf \big) \quad 
    \text{ in } (0,T) \times \Omega \times \ConfigurationSpace, 
\end{multline}
where $\comdiff$ denotes the center-of-mass diffusion coefficient and the Deborah number, $\Deborah$, accounts for the relaxation time of the polymer molecules.

In the Navier--Stokes--Fokker--Planck (NSFP) system, the microscopic and macroscopic dynamics are coupled through the velocity field $\Velocity$ and velocity gradient $\Grad_{\Coordinate} \Velocity$ in the FP equation and the probability density $\FPpdf$ enters in the extra-stress tensor in the NS equations, in the form of Kramer's expression 
\begin{align}
    \label{eq:KramersExpression}
    \ExtraStress(\Time, \Coordinate) := 
    \ExtraStress(\FPpdf)  = 
    \gamma \int_\ConfigurationSpace \left( \SpringForce \Configuration \Transpose  - \Id \right) \FPpdf(\Time, \Coordinate, \Configuration) \,\Diff \Configuration, 
\end{align}
where the dimensionless quantity $\gamma = (1- \beta) / (\Deborah\, \Reynolds)$. It accounts for the viscosity contribution of the polymer molecules and their relaxation time.

While most studies assume that the stress response depends only on the current deformation rate, a dependency on the deformation history is no novel idea.
Already in 1983, Bagley and Torvik \cite{bagley1983theoretical} reinterpreted the stress from solute polymer molecules in the Rouse model \cite{rouse1953theory} as the fractional derivative of the polymer molecules' deformation history.
To account for more complex time-dependent behaviors of the mixture, \cite{fritz2024analysis} recently extended the model for dilute polymeric fluids to include trapping effects of the polymer molecules, such as configuration changes, in the form of a Riemann--Liouville time-fractional derivative.
For $\alpha \in (0,1]$, the Riemann--Liouville time-fractional derivative is defined as
\begin{align}
    \label{eq:RiemannLiouvilleFractionalDerivative}
    \tf f := \PDiff{\Time} \int_0^\Time \tfkernel(\Time - s) f(s) \,\Diff s, \qquad 
    \tfkernel(\Time) := \frac{\Time^{\alpha-1}}{\Gamma(\alpha)},
\end{align}
where $\Gamma(s) = \int_0^\infty t^{s-1} e^{-t} \,\Diff t$ denotes the Euler gamma function.
The derivation replaces the probability density $\FPpdf$ on the right-hand side of equation \eqref{eq:FokkerPlanck} and in the definition of the extra-stress tensor \eqref{eq:KramersExpression} with its time-fractional derivative $\tf \FPpdf$.
Thus, the time-fractional Fokker--Planck (TFFP) equation reads as
\begin{align}
    \label{eq:timefractionalFokkerPlanck}
    \PDiff{\Time} \FPpdf 
    = 
    & - \left( \Velocity \cdot \Grad_{\Coordinate} \right) \tf \FPpdf 
    + \comdiff \Laplace_{\Coordinate} \tf \FPpdf \\ \nonumber
    & - \Grad_{\Configuration} \cdot \left((\Grad_{\Coordinate} \Velocity) \Configuration \tf \FPpdf \right) 
    + \frac{1}{2 \Deborah} \Grad_{\Configuration} \cdot \left(\SpringForce \tf \FPpdf + \Grad_{\Configuration} \tf \FPpdf \right)  
    \qquad \text{ in } (0,T) \times \Omega \times \ConfigurationSpace.
\end{align}
As discussed in \cite{heinsalu2007use}, this is the correct placement of the time-fractional derivative for a time-dependent force, while a Caputo-type time-fractional derivative on the left-hand side of the equation would be no physical representation of the underlying stochastic process.
In \cite{fritz2024analysis}, the well-posedness of the simplified corotational time-fractional NSFP system with a FENE spring force was shown, but only for $\alpha \in (1/2,1)$.

To avoid the history dependence before $t=0$ and the associated singular behavior of time-fractional derivatives, for the time-fractional NSFP system, we restrict ourselves to starting from rest, choosing the initial data as 
\begin{equation}
    \Velocity\big|_{\Time=0} \equiv 0 \quad  \text{ in } \CoordinateSpace, \qquad 
    \FPpdf(0, \Coordinate, \Configuration) = c \exp\left(-\frac{\Abs{\Configuration}^2}{2}\right)\quad  \text{ in } \CoordinateSpace \times \ConfigurationSpace, 
\end{equation}
whereby $c$ is a normalizing constant such that $\int_\ConfigurationSpace \FPpdf(0, \Coordinate, \Configuration) \,\Diff \Configuration \equiv 1$.
For the NS equations, we consider mixed boundary conditions. 
Let $\Gamma_D$ and $\Gamma_N$ denote a partition of the boundary $\partial\Omega$ into the Dirichlet and Neumann part, respectively, i.e., $\partial \CoordinateSpace = \Gamma_D \cup \Gamma_N$, $\Gamma_D \cap \Gamma_N = \emptyset$, then  
\begin{alignat}{3}
    \label{eq:NS_BC}
    \Velocity 
    & = \Velocity_D \quad
    && \text{ on } (0,T) \times \Gamma_D, \\ 
    \Grad_{\Coordinate} \Velocity \vec{n} 
    & = \vec{0} \quad
    && \text{ on } (0,T) \times \Gamma_{N},
\end{alignat}
where $\vec{n}$ denotes the outer normal vector on $\CoordinateSpace$. We assume that the Dirichlet boundary condition is compatible with the initial condition.
On the boundary of the spatial domain, we equip the TFFP equation with homogeneous Neumann boundary conditions
\begin{equation}
    \label{eq:FP_BC}
    \Grad_{\Coordinate} \FPpdf \cdot \vec{n} = 0, \quad \text{ on } (0,T) \times \partial\CoordinateSpace \times \ConfigurationSpace, 
\end{equation}
which corresponds to the assumption of having a uniform particle distribution over the spatial domain, see, e.g., Lemma 3.2 in \cite{barrett2010existence}.
For the Hookean spring force, we consider sufficiently strong decay conditions for $\FPpdf$ in the configuration space as $|\Configuration| \rightarrow \infty$, see, e.g., \cite{barrett2018existence}.

Direct numerical simulations of the fully coupled NSFP system are extremely challenging even for the integer-order case since the FP equation is defined on the Cartesian product of two $d$-dimensional spaces.
To overcome this problem, Chauviere and Lozinski \cite{chauviere2004simulation} suggested the application of a first-order Strang splitting \cite{strang1968construction} to solve the problems on the spatial domain and the configuration space separately, see also \cite{cao2015time} for operator splittings of time-fractional differential equations.
By applying a spectral method to the configuration space, one transforms the $2d$-dimensional system into a large coupled system of partial differential equations (PDEs) defined on the spatial domain.
This approach was applied to the FENE-type NSFP system using spherical harmonics and Jacobi-polynomials in \cite{Knezevic_Suli_2009} and to a Hookean-type model using the Hermite spectral method by \cite{mizerova2018conservative} and later to time-fractional Hookean-type NSFP by \cite{beddrich2024numerical}.
Due to the large number of coupled PDEs and the high temporal and spatial resolution required for turbulent flow, the drag-reducing effect of dilute polymeric fluids has not yet been rigorously studied.
Modern stochastic approaches to approximate the micro-macro system applying the Brownian configuration field method, \cite{GRIEBEL201441,ruttgers2018multiscale}, see also \cite{chauviere2002new,chauviere2003efficient}, variance reduction methods such as the multilevel Monte Carlo method \cite{ye2018numerical}, even for highly parallelized GPU implementations \cite{cromer2023macro} also suffer from the
high computational cost and are limited to a few prototypical test cases.

To avoid the challenges posed by high-dimensional PDE systems, macroscopic models describing the extra-stress tensor using a fully macroscopic constitutive relation are widely used for the numerical simulation of drag-reducing agents.
Considering a linear spring force, multiplying the FP equation by $(\SpringForce \Configuration \Transpose  - \Id)$ and integrating over the configuration space results in the diffusive Oldroyd-B model. For an overview of the mathematical intricacies of the Oldroyd-B model \cite{oldroyd1950formulation}, we refer to the recent review of Renardy and Thomases \cite{renardy2021mathematician}.
For the FENE model, it is impossible to analytically derive such a constitutive relation for the extra-stress tensor \cite{herrchen1997detailed}.
Still, macroscopic approximations of the FENE model can be obtained based on simplifications of the spring force. 
For instance, the Peterlin closure, replacing the term $\Abs{\Configuration}^2$ in the FENE spring force with $\int_{\ConfigurationSpace} \Abs{\Configuration}^2 \FPpdf \,\Diff\Configuration$, results in the popular FENE-Peterlin (FENE-P) model. 
Further macroscopic models relevant in the context of polymeric fluids include the Upper-Convected-Maxwell model \cite{Lahiri2023drag}, the linear Jeffreys model \cite{dellar2014lattice}, the Phan--Thien--Thanner model \cite{sousa2011effect}, and the Giesekus model \cite{yu2004numerical,DIMITROPOULOS1998433}. 
We refer the interested reader to \cite{alves2021numerical} for a general overview of numerical methods for viscoelastic fluids. 
To our knowledge, there is no study considering polymer molecules with memory effects in the context of drag-reducing agents.

This article is organized as follows:
In Section \ref{sec:preliminaries}, we introduce the relevant mathematical preliminaries, namely the Hermite spectral method and the kernel compression method based on rational approximation for the time-fractional derivative.
In Section \ref{sec:macroscopic_model}, we derive a macroscopic model for the extra-stress tensor of the time-fractional Hookean-type NSFP system, which is equivalent to solving the fully coupled micro-macroscopic model with the Hermite spectral method, building on \cite{beddrich2024numerical}.
Thereafter, in Section \ref{sec:numerical_scheme}, we state the semi-discrete system and derive the fully discrete formulation of the problem.
Section \ref{sec:numerical_experiments} contains numerical experiments showing the convergence order of the approach, the turbulence-reducing effect of the polymer molecules, the influence of the number of polymer molecules, and the Deborah and Reynolds number, as well as the influence of the order of the time-fractional derivative.
We conclude the article in Section \ref{sec:conclusion}.

\section{Preliminaries}
\label{sec:preliminaries}
We will use the Hermite spectral method to discretize the FP equation over the configuration space.
Therefore, we introduce the method and the corresponding fundamentals in \ref{subsec:Hermite_spectral_method}.
Similarly, we outline the foundations of the kernel compression method in Section \ref{subsec:kernel_compression_method}, which we will apply to the time-fractional derivative.
For more detailed introductions, we refer the interested reader to the textbooks of Shen \cite{shen2011spectral} and Diethelm \cite{diethelm2010analysis}, respectively.

\subsection{Hermite spectral method}
\label{subsec:Hermite_spectral_method}

The Hermite polynomials on $\R$ are defined by    
\begin{equation}
    \HermitePolynomial{m}(r) := \left(-1\right)^m e^{r^2} \partial_r^m \left(e^{-r^2}\right), \quad m \in \N_0.
\end{equation}
Equivalently, the Hermite polynomials can be introduced as a $w$-orthogonal family of polynomials for the weight function $w(r):= e^{-r^2}$.
Since the Hermite polynomials $H_0$, $H_1$, and $H_2$ will be of utmost importance in this work, we state them explicitly 
\begin{equation}
    \HermitePolynomial{0}(x) = 1, \qquad 
    \HermitePolynomial{1}(x) = 2x, \qquad 
    \HermitePolynomial{2}(x) = 4x^2-2.
\end{equation}
For $a > 0$, the scaled Hermite functions are defined as 
\begin{equation}
    \HermiteFunction{m}(r):= \left(\frac{\sqrt{\scaling}}{\sqrt[4]{\pi}}\right) \frac{w(\scaling r)}{\sqrt{2^m m!}} \HermitePolynomial{m}(\scaling r), 
\end{equation}
where the prefactor $\sqrt{\scaling} / \sqrt[4]{\pi}$ is non-standard and solely introduced to obtain orthonormality in Equation \eqref{eq:Hermite_functions_orthonormality}.
Considering a function decaying towards infinity, the scaling parameter $\scaling$ is commonly chosen to increase the accuracy of the approximation with Hermite functions within the effective interval outside of which the function is negligible.
To ensure an accurate calculation of the extra-stress tensor, we will set $\scaling = 1/\sqrt{2}$, see Section \ref{subsec:mm_scaling_parameter}.
The Hermite polynomials and Hermite functions up to order five are displayed in Figure \ref{fig:Hermite_polynomials_functions}.
\begin{figure}
     \centering
     \includegraphics[width=0.95\textwidth]{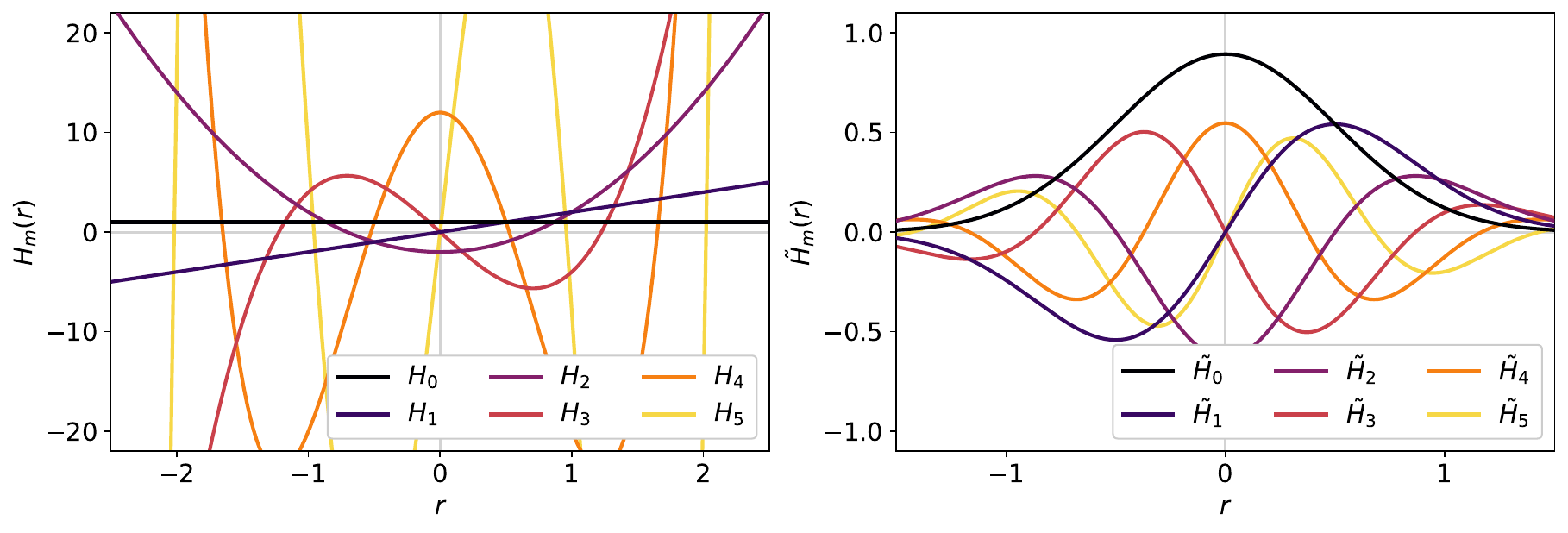}
     \caption{Hermite polynomials (left) and Hermite functions with scaling parameter $a = 1 / \sqrt{2}$ (right) up to order 5.}
     \label{fig:Hermite_polynomials_functions}
\end{figure}

The following properties of scaled Hermite functions are essential for the practicability of the Hermite spectral method.
The Hermite functions form an orthonormal basis of $L^2_{1/w}(\R)$ in the sense: 
\begin{equation}
    \label{eq:Hermite_functions_orthonormality}
    \int_\R \frac{\HermiteFunction{n}(r) \HermiteFunction{m}(r)}{ w(\scaling r)} \,\Diff r = \delta_{nm}, \quad 
    n,m \in \N_0.
\end{equation}
The only Hermite function with mass is $\HermiteFunction{0}$, i.e., for $n \in \N$, it holds  
\begin{equation}
    \label{eq:Hermite_functions_zero_mass}
    \int_\R \HermiteFunction{n}(r) \,\Diff r = 0, \quad 
    n \in \N.
\end{equation}
Finally, the derivative of a Hermite function and the product of a Hermite function with a linear function can be expressed in terms of Hermite functions.
For $m \in \N_0$, 
\begin{align}
    \label{eq:Hermite_functions_derivative}
    \HermiteFunction{m}'(r) &= -\scaling \sqrt{2(m+1)} \HermiteFunction{m+1}(r), \\ 
    \label{eq:Hermite_functions_times_r}
    \scaling r \HermiteFunction{m}(r) & = \sqrt{\frac{m+1}{2}} \HermiteFunction{m+1}(r) + \sqrt{\frac{m}{2}} \HermiteFunction{m-1}(r), 
\end{align}
where we use the convention $\HermiteFunction{-1} \equiv 0$.

The starting point of our numerical scheme is the approximation of $\FPpdf$ over the configuration space $\ConfigurationSpace$ by a tensor product of Hermite functions.
This choice is inherent, as $\ConfigurationSpace = \R^d$ is unbounded, and the approximation with Hermite functions automatically decays exponentially toward infinity. 
For $N \in \N_0$, the $L^2_w(\R)$-orthogonal projection $\Pi_N : L^2_w(\R) \rightarrow P_N$ onto the space of polynomials up to degree $N$ is defined by 
\begin{equation}
    \label{eq:Hermite_projection}
    \left(f - \Pi_N f, g \right)_w := \int_\R \big(f(y) - (\Pi_Nf)(y)\big) g(y) w(y) \Diff y = 0, \quad \forall g \in P_N. 
\end{equation}
The Hermite function approximation of any $f \in L^2_{1/w}(\R)$ is given as $\Tilde{\Pi}_N f := w \Pi_N (f/w) \in \text{span}\{ w \HermitePolynomial{k}\}_{k=0}^N$, and thus, we define the approximation of $\FPpdf$ with respect to Hermite functions with scaling parameter $a$ as
\begin{equation}
    \label{eq:HSM_psiN_definition}
    \FPpdf_N(\Configuration) 
    := \sum_{i_1, \ldots, i_d=0}^N \HermiteModes{i_1, \ldots, i_d} \HermiteFunction{i_1, \ldots, i_d}(\Configuration),  
\end{equation}
whereby we make use of the notation $\HermiteFunction{i_1, \ldots, i_d}(\Configuration) = \prod_{k=1}^d\HermiteFunction{i_k}(q_k)$, and the so-called Hermite spectral modes $\HermiteModes{i_1, \ldots, i_d}$ are defined as 
\begin{equation}
    \label{eq:HSM_phi_definition}
    \HermiteModes{i_1, \ldots, i_d}(\Time, \Coordinate):= \int_{\ConfigurationSpace} \FPpdf(\Time, \Coordinate, \Configuration) \HermiteFunction{i_1, \ldots, i_d}(\Configuration) \prod_{k=1}^d w(\scaling q_k)^{-1} \,\Diff\Configuration.
\end{equation}
We denote by $\sum_{k=1}^d i_k$ the degree of the Hermite spectral mode $\HermiteModes{i_1, \ldots, i_d}$ and the Hermite function $\HermiteFunction{i_1, \ldots, i_d}$, respectively.

\subsection{Kernel compression method}
\label{subsec:kernel_compression_method}

To discretize the time-fractional derivative, we use a kernel compression method, i.e., by approximating the time-fractional integral kernel with a weighted sum of exponentials, we transform the time-fractional differential equation into a system of ODEs.
The approach considered in this work is based on the work of \cite{khristenko2023solving}.
The time-fractional integral kernel is approximated as the sum of exponential terms 
\begin{equation}
    \label{eq:tf_approximate_kernel}
    \tfkernel(t) \approx \tfapproxkernel(t)
    := \sum_{k=1}^{m} \tfweights e^{- \tfpoles t}, 
\end{equation}
whereby $\delta(t)$ denotes the Dirac-delta distribution.
By Bernstein's theorem for completely monotone functions \cite{bernstein1929fonctions}, the approximation $\tfapproxkernel$ of the integral kernel is completely monotone if, for $k = 1, \ldots, m$, the weights $\tfweights$ and the poles $\tfpoles$ are non-negative. Replacing the fractional integral kernel $\tfkernel$ with its approximation $\tfapproxkernel$ in the definition of the time-fractional derivative, for $\Time > 0$, we obtain
\begin{equation}
    \label{eq:fractional_derivative_as_modes}
    \PDiff{\Time} (\tfapproxkernel * f) (\Time) = \PDiff{\Time} \sum_{k=1}^m f_k(\Time), 
\end{equation}
whereby we define the fractional modes $f_k$ as 
\begin{align}
    \label{eq:fractional_modes_definition}
    f_k(\Time) := \tfweights \int_0^\Time e^{- \tfpoles (\Time - s)} f(s) \,\Diff s, \qquad k=1,\ldots,m.
\end{align}
We observe, that for $k=1,\ldots,m$, the fractional modes $f_k$ are solutions of the initial value problems
\begin{equation}
    \label{eq:fractional_modes_equations}
    \PDiff{\Time} f_k = - \tfpoles f_k + \tfweights f, \qquad f_k(0) = 0, 
\end{equation}
where the initial condition of the mode equations arises directly from the definition in \eqref{eq:fractional_modes_definition}, as for a sufficiently smooth $f$ the integral value tends to zero as $t$ does.
Inserting \eqref{eq:fractional_modes_equations} in \eqref{eq:fractional_derivative_as_modes}, we rewrite the approximation of the time-fractional derivative as 
\begin{equation}
    \label{eq:fractional_modes_fractional_derivative}
    \PDiff{\Time} (\tfapproxkernel * f) 
    = \sum_{k=1}^m \left(- \tfpoles f_k + \tfweights f\right), 
\end{equation}
allowing us to approximate the time-fractional differential equation by a system of ODEs.

Obtaining such an approximation of the fractional integral kernel of the form \eqref{eq:tf_approximate_kernel} is a non-trivial problem. Khristenko et al. \cite{khristenko2023solving} suggest using the rational approximation of $s^{-\alpha}$, the Laplace transform of the integral kernel, employing the adaptive Anatoulas Andersen (AAA) algorithm \cite{nakatsukasa2018aaa}, whereby both numerator and denominator have the same degree. Subsequent partial fraction decomposition and application of the inverse Laplace transform results in an approximation of the form 
\begin{equation}
    \label{eq:tf_approximate_kernel_khristenko}
    \tfkernel(t) \approx \tfapproxkernel(t)
    := \sum_{k=1}^{m} \tfweights e^{- \tfpoles t} + \tfwinf \delta(t), 
\end{equation}
where $\delta(t)$ is the Dirac delta function. For $w_\infty > 0$, naively replacing the time-fractional kernel with \eqref{eq:tf_approximate_kernel} violates the property 
\begin{equation}
    \lim_{\Time \rightarrow 0^+}\int_0^\Time \tfkernel(\Time - s) f(s) \,\Diff s = 0,
\end{equation} 
which holds for $\alpha > 0$, $p > \max\{1,\, 1/\alpha\}$ and $f \in L^p(0,T)$, see \cite{diethelm2010analysis}.
To overcome this inconsistency, we follow the approach introduced by Duswald et al. \cite{duswald2024finite} for the fractional Laplacian. Applying the AAA algorithm to $s^{1-\alpha}$ and dividing by $s$, we obtain a rational approximation of $s^{-\alpha}$, whereby the degree of the numerator is by one lower than the degree of the denominator, resulting in an approximation
of the kernel as in \eqref{eq:tf_approximate_kernel_khristenko} with $\tfwinf =0$. Thus, we will use this approach as it guarantees automatically the approximation of the kernel as in 
\eqref{eq:tf_approximate_kernel}.

\section{Derivation of a macroscopic formulation}
\label{sec:macroscopic_model}
Discretizing the FP equation over the configuration space $\ConfigurationSpace$ with spectral methods results in large coupled PDE systems of the spectral modes over the spatial domain $\CoordinateSpace$, and thus, fully macroscopic models.
As the corresponding configuration space is unbounded, the Hermite spectral method is an inherent choice for the Hookean spring model.

Mizerova and She \cite{mizerova2018conservative} applied the Hermite spectral method to the Hookean-type NSFP equations, whereby a first-order Strang splitting was used to solve the FP equation separately over the spatial domain and the configuration space.
We extended this approach by exploiting that spectral modes are independent of higher-order spectral modes, 
and that only second-order Hermite spectral modes are necessary to calculate the polymer-induced extra stresses in the NS equations \cite{beddrich2024numerical}; see also \cite{hetland2023solving}.
Based on these findings, we derive a macroscopic model for the extra-stress tensor with only $7$ ($4$ in 2D) equations, equivalent to solving the Hookean-type NSFP system with the Hermite spectral method.
We derive the macroscopic model for the integer-order FP equation and restrict ourselves to the three-dimensional case for readability.

\subsection{Extra-stress tensor}
\label{subsec:mm_extra_stress}

We calculate the extra-stress tensor in the NS equations for the Hermite function approximation of $\FPpdf$.
Inserting $\FPpdf_N$ from Equation \eqref{eq:HSM_psiN_definition} into Equation \eqref{eq:KramersExpression}, we obtain 
\begin{align}
    \label{eq:ExtraStressN_definition}
    \ExtraStress_N  (\Time, \Coordinate) 
    & := \ExtraStress(\FPpdf_N) 
    = \gamma 
    \sum_{j,k,l = 0}^N \HermiteModes{j,k,l} (\Time, \Coordinate) \int_\mathcal{D} \left(\Configuration \Configuration\Transpose - \Id\right) \HermiteFunction{j,k,l} (\Configuration) \,\Diff\Configuration.
\end{align}
Since Hermite functions with odd degrees are odd functions and exploiting the $w$-orthogonality of the Hermite polynomials, only the integrals with Hermite functions of degree $0$ or $2$ are nonzero.
Thus, for $N \geq 2$, the extra-stress tensor can be expressed as 
\begin{align}
    \label{eq:ExtraStressN_0and2order}
    \ExtraStress_N
    & = 
    \gamma \left(\begin{pmatrix}
        \chi_2 \phi_{2,0,0} & \chi_1 \phi_{1,1,0} & \chi_1 \phi_{1,0,1} \\ 
        \chi_1 \phi_{1,1,0} & \chi_2 \phi_{0,2,0} & \chi_1 \phi_{1,1,0} \\ 
        \chi_1 \phi_{1,0,1} & \chi_1 \phi_{0,1,1} & \chi_2 \phi_{0,0,2}
    \end{pmatrix}
    - \chi_0 \phi_{0,0,0} \Id
    \right), 
\end{align}
where the coefficients are given by 
\begin{alignat}{3}
    \label{eq:chi_integrals_Hermite_functions}
    \chi_0 
    & = \int_{\ConfigurationSpace} (q_1^2-1) \HermiteFunction{0,0,0}(\Configuration) \,\Diff\Configuration 
    && = \frac{(1-2\scaling^2) \pi^{3/4}}{2 \scaling^{7/2}}, \\  
    \chi_1 
    & = \int_{\ConfigurationSpace} q_1 q_2 \HermiteFunction{1,1,0}(\Configuration) \,\Diff\Configuration
    && = \frac{\pi^{3/4}}{2 \scaling^{7/2}},\\ 
    \chi_2 
    & = \int_{\ConfigurationSpace} q_1^2 \HermiteFunction{2,0,0}(\Configuration) \,\Diff\Configuration 
    && = \frac{\pi^{3/4}}{\sqrt{2} \scaling^{7/2}}.
\end{alignat}
It is worth pointing out that for $N\geq2$, the extra-stress tensor depends only on spectral modes of degrees $0$ and $2$. Note that for $\scaling = 1 / \sqrt{2}$, $\chi_0 = 0$; thus, the zero-degree spectral mode does not contribute to the coupling tensor.

\subsection{Fokker--Planck equation}
\label{subsec:mm_fokker_planck}

Denoting the spatio-temporal differential operator and the configuration differential operator in the FP equation by
\begin{align}
    \mathcal{X}(\Velocity) \FPpdf
    & := \PDiff{\Time} \FPpdf + \left(\Velocity \cdot \nabla_{\Coordinate} \right) \FPpdf - \comdiff \Laplace_{\Coordinate} \FPpdf, \\ 
    \mathcal{Q}(\Grad_{\Coordinate}\Velocity)\FPpdf 
    & := - \Grad_{\Configuration} \cdot \big((\Grad_{\Coordinate}\Velocity) \Configuration \FPpdf \big)
    + \frac{1}{2 \Deborah} \Grad_{\Configuration} \cdot \big(\Configuration \FPpdf + \Grad_{\Configuration}\FPpdf \big),
\end{align}
the FP equation reads as 
\begin{equation}
    \label{eq:FokkerPlanckXQ}
    \mathcal{X}(\Velocity) \FPpdf = \mathcal{Q}(\Grad_{\Coordinate}\Velocity) \FPpdf.
\end{equation}
We insert $\FPpdf_N$ into \eqref{eq:FokkerPlanckXQ}, multiply with the test function 
\begin{align}
    \label{eq:test_function_FP}
    \TestFunction_{r,s,z} (\Configuration) = \frac{\HermiteFunction{r,s,z}(\Configuration)}{\weight(\scaling q_1) \weight(\scaling q_2) \weight(\scaling q_3)}, 
\end{align}
whereby $0 \leq r,s,z, \leq N$, and integrate over the configuration space.
Since the spatio-temporal operator does not depend on the configuration $\Configuration$, we obtain  
\begin{equation}
    \label{eq:space_time_op_tested}
    \int_{\ConfigurationSpace} \mathcal{X}(\Velocity) \FPpdf_N \TestFunction_{r,s,z} \,\Diff\Configuration 
    = \mathcal{X}(\Velocity) \HermiteModes{r,s,z}.
\end{equation}
Proceeding analogously for the configuration operator, following the steps outlined in \cite{mizerova2018conservative}, we find 
\begin{align}
    \label{eq:conf_op_tested}
    & \int_{\ConfigurationSpace} \mathcal{Q}(\Grad_{\Coordinate}\Velocity) \FPpdf_N \TestFunction_{z_1,z_2,z_3} \,\Diff\Configuration 
    \\ \nonumber
    & = \HermiteModes{z_1,z_2,z_3} \sum_{l=1}^d \left((\Grad_{\Coordinate}\Velocity)_{ll} - \frac{1}{2 \Deborah} \right) z_l 
    + \sum_{l=1}^d \HermiteModes{(z_l - 2)} \left(\frac{2 \scaling^2 - 1}{2 \Deborah} + (\Grad_{\Coordinate}\Velocity)_{ll} \right) \sqrt{(z_l-1) z_l}
    \\ \nonumber
    &\quad + \sum_{1 \leq l \neq m \leq d} \HermiteModes{(z_l + 1, z_m - 1)} (\Grad_{\Coordinate}\Velocity)_{ml} \sqrt{(z_l+1) z_m}
    + \sum_{1 \leq l \neq m \leq d} \HermiteModes{(z_l - 1, z_m - 1)} (\Grad_{\Coordinate}\Velocity)_{ml} \sqrt{z_l z_m}, 
\end{align}
where we introduced the notation $(z_l-2)$ to indicate that the index $z_l$ is replaced by $z_l - 2$, while all other indices remain the same.
Analogously, $(z_m +1, z_l-1)$ denotes the replacement of $z_m$ and $z_l$ with $z_m+1$ and $z_l-1$, respectively.
We set the corresponding Hermite spectral mode to zero if any index is negative or larger than $N$. 

We point out two key observations previously outlined in \cite{beddrich2024numerical}.
In \eqref{eq:conf_op_tested}, every Hermite spectral mode $\HermiteModes{j,k,l}$ depends on itself and, at most, 12 other Hermite spectral modes.
Further, every $\HermiteModes{j,k,l}$ only depends on Hermite spectral modes, whose degree is $j+k+l$ or $j+k+l-2$, and thus, not Hermite spectral modes with higher degree. 
Similar approaches are unavailable for polynomial spring forces of a degree larger than one, as this would result in a dependency on Hermite spectral modes with a higher degree.

\subsection{The choice of the scaling parameter}
\label{subsec:mm_scaling_parameter}

In previous works \cite{mizerova2018conservative,beddrich2024numerical}, applying the Hermite spectral method to Hookean-type NSFP equations, the scaling parameter was arbitrarily chosen as $\scaling = 1/2$.
However, since the accuracy of approximations with Hermite functions is highly dependent on the choice of the scaling parameter \cite{mohammadi2015hermite}, we further investigate the role of the scaling parameter.
By considering homogeneous flow, i.e., $\Grad_{\Coordinate} \Velocity$ is symmetric, we construct an analytical solution of the FP equation and find that by exactly choosing $\scaling = 1 / \sqrt{2}$, the Hermite approximate of the extra-stress tensor and the analytical extra-stress tensor coincide.

Assuming that $\FPpdf$ is constant over the spatial domain $\CoordinateSpace$, the FP equation \eqref{eq:FokkerPlanckXQ} simplifies to 
\begin{align}
    \label{eq:FP_simplified_constant_in_space}
    \PDiff{\Time} \FPpdf = \mathcal{Q}(\Grad_{\Coordinate}\Velocity) \FPpdf.
\end{align}
Since we consider a homogeneous flow scenario, the gradient of the velocity field is symmetric and trace-free.
With this, we can determine a class of functions in the kernel of the configuration operator and, thus, a class of steady-state solutions of the FP equation.

\begin{lemma}
    Let $\tensor{D} \in \R^{3 \times 3}$ be a trace-free, symmetric matrix, $\Deborah \in \R$, such that 
    \begin{align}
        \tensor{C} := \frac{1}{2} \Id - \Deborah \, \tensor{D} 
    \end{align}
    is symmetric and positive definite.
    Then 
    \begin{align}
        \mathcal{Q}(\tensor{D}) \exp\left(- \Configuration\Transpose \tensor{C} \Configuration \right)
        = 0.
    \end{align}
\end{lemma}

\begin{proof}
    We define $f(\Configuration) := \exp{\left(- \Configuration\Transpose \tensor{C} \Configuration \right)}$, and the result follows from the following direct calculation:
    \begin{align*}
        \mathcal{Q}(\tensor{D}) f(\Configuration)  
        &= - \Grad_{\Configuration} \cdot \big(\tensor{D} \Configuration f(\Configuration) \big)
        + \frac{1}{2\Deborah} \Grad_{\Configuration} \cdot \big(\Configuration f(\Configuration) + \Grad_{\Configuration}f(\Configuration) \big) \\
        &= f(\Configuration) \left(
        \Configuration\Transpose \left( 
            2 \tensor{D} \tensor{C}  
            - \frac{2}{2\Deborah} \tensor{C} 
            + \frac{4}{2\Deborah} \tensor{C}^2 
        \right) \Configuration
        - \Trace\left(\tensor{D}\right) 
        + \frac{3}{2\Deborah} 
        - \frac{4}{2\Deborah} \Trace \left(\tensor{C}\right)
        \right) \\ 
        &= 0.
        \qedhere
    \end{align*}
\end{proof}
Having established a class of analytical steady-state solutions of \eqref{eq:FP_simplified_constant_in_space}, we calculate the analytical extra-stress tensor $\ExtraStress$ and the approximated extra-stress tensor $\ExtraStress_2$.

\begin{lemma}
    \label{lem:tau_equals_tau2}
    Let $\tensor{C} \in \R^{3 \times 3}$ be symmetric and positive definite, $f(\Configuration) = \exp{\left(- \Configuration\Transpose \tensor{C} \Configuration \right)}$.
    Then the best $L^2_{1/w}(\R)$ approximation $f_N$ of $f$ using Hermite functions with scaling parameter $\scaling = 1 / \sqrt{2}$ up to degree $N \in \N_0$ is defined as 
    \begin{equation}
        \label{eq:f2}
        f_N(\Configuration) 
        = \sum_{j, k, l = 0}^N \HermiteFunction{j,k,l}(\Configuration) \int_{\ConfigurationSpace} f(\Configuration) \HermiteFunction{j,k,l}(\Configuration) \prod_{i=1}^3 w(\scaling q_i)^{-1} \,\Diff\Configuration
    \end{equation}
  and satisfies for $N \geq 2$
    \begin{equation}
        \label{eq:tau_equals_tauN}
        \int_\ConfigurationSpace \left(\Configuration \Configuration \Transpose - \Id \right) f(\Configuration) \,\Diff \Configuration 
        = \int_\ConfigurationSpace \left(\Configuration \Configuration \Transpose - \Id \right) f_N (\Configuration) \,\Diff \Configuration.
    \end{equation}
\end{lemma}

\begin{proof} 
    See \ref{apdx:proof_lemma}.
\end{proof}

\subsection{Macroscopic system}

The results of Sections \ref{subsec:mm_extra_stress}, \ref{subsec:mm_fokker_planck}, and \ref{subsec:mm_scaling_parameter} allow us to formulate a fully macroscopic model equivalent to solving the Hookean-type NSFP system with the Hermite spectral method. 
We choose $\scaling = 1 / \sqrt{2}$ and denote the vector of the Hermite spectral modes 
\begin{equation}
    \Vec{\Phi} = (\HermiteModes{0,0,0}, \HermiteModes{0,1,1}, \HermiteModes{1,0,1}, \HermiteModes{1,1,0}, \HermiteModes{0,0,2}, \HermiteModes{0,2,0}, \HermiteModes{2,0,0}){\Transpose}.
\end{equation} 
The fully macroscopic time-fractional NSFP model reads as  
\begin{alignat}{3}
    \label{eq:macroscopic_fokker_planck}
    \PDiff{\Time} \Vec{\Phi}
    + \big(\left(\Velocity \cdot \Grad_{\Coordinate} \right)
    - \comdiff \Laplace_{\Coordinate}
    - \tensor{A}(\Grad_{\Coordinate} \Velocity) \big) \tf \Vec{\Phi} 
    & = 0 
    \quad && \text{ in } (0,T) \times \Omega, \\ 
    \PDiff{\Time} \Velocity
    + \left( \Velocity \cdot \Grad_{\Coordinate} \right) \Velocity 
    + \Grad_{\Coordinate} \Pressure 
    - \frac{\beta}{\Reynolds} \Laplace_{\Coordinate} \Velocity 
    - \Grad_{\Coordinate} \cdot \ExtraStress \left(\tf \Vec{\Phi} \right)
    & = 0
    \quad && \text{ in } (0,T) \times \Omega, \\ 
    \Grad_{\Coordinate} \cdot \Velocity 
    & = 0 
    \quad && \text{ in } (0,T) \times \Omega, 
\end{alignat}
where the spatial operators are applied entry-wise to $\tf \Vec{\Phi}$. From now on, we refer to \eqref{eq:macroscopic_fokker_planck} also as the macroscopic FP (MFP) equation. 
The coupling between the entries of $\tf \Vec{\Phi}$ is given by
\begin{align}
    \tensor{A}(\Grad_{\Coordinate}\Velocity) = 
    \begin{pmatrix}
        0 
        & \vec{0}\Transpose \\ 
        \Vec{K}(\Grad_{\Coordinate}\Velocity) 
        & \tensor{A}_2 (\Grad_{\Coordinate}\Velocity) - \frac{1}{2 \Deborah} \Id
    \end{pmatrix}, 
\end{align}
where $\Vec{K}(\Grad_{\Coordinate}\Velocity) \in \R^6$, the contribution of the zeroth-order spectral mode, is given by 
\begin{align}
    \label{eq:mm_contribution_order0}
    \Vec{K}(\ConfAdvectionCoeff)
    & = \begin{pmatrix}
        \left(\kappa_{23} + \kappa_{32}\right), &
        \left(\kappa_{13} + \kappa_{31}\right), &
        \left(\kappa_{12} + \kappa_{21}\right), &
        \kappa_{33}, &
        \kappa_{22}, &
        \kappa_{11}
    \end{pmatrix}\Transpose,   
\end{align}
and $\tensor{A}_2(\Grad_{\Coordinate}\Velocity) \in \R^{6 \times 6}$, the coupling between the second-order spectral modes, reads as 
\begin{align}
    \label{eq:mm_coupling_order2}
    \tensor{A}_2(\ConfAdvectionCoeff) 
    & = \begin{pmatrix}
        \left(\kappa_{22} + \kappa_{33}\right) & \kappa_{21} & \kappa_{31} & \sqrt{2} \kappa_{23} & \sqrt{2} \kappa_{32} & 0 \\ 
        \kappa_{12} & \left(\kappa_{11} + \kappa_{33}\right) & \kappa_{32} & \sqrt{2} \kappa_{13} & 0 & \sqrt{2} \kappa_{31} \\ 
        \kappa_{13} & \kappa_{23} & \left(\kappa_{11} + \kappa_{22}\right) & 0 & \sqrt{2} \kappa_{12} & \sqrt{2} \kappa_{21} \\ 
        \sqrt{2} \kappa_{32} & \sqrt{2} \kappa_{31} & 0 & 2\kappa_{33} & 0 & 0 \\ 
        \sqrt{2} \kappa_{23} & 0 & \sqrt{2} \kappa_{21} & 0 & 2\kappa_{22} & 0 \\ 
        0 & \sqrt{2} \kappa_{13} & \sqrt{2} \kappa_{12} & 0 & 0 & 2\kappa_{11} 
    \end{pmatrix}.
\end{align}
As derived in Section \ref{subsec:mm_extra_stress}, the extra-stress tensor in the NS equations is determined by
\begin{align}
    \label{eq:mm_ExtraStress}
    \ExtraStress(\tf \tensor{\Phi})
    = (2\pi)^{3/4} \gamma \tf 
    \begin{pmatrix}
        \sqrt{2} \HermiteModes{2,0,0}   & \HermiteModes{1,1,0}          & \HermiteModes{1,0,1} \\ 
        \HermiteModes{1,1,0}            & \sqrt{2} \HermiteModes{0,2,0} & \HermiteModes{0,1,1} \\           
        \HermiteModes{1,0,1}            & \HermiteModes{0,1,1}          & \sqrt{2} \HermiteModes{0,0,2}
    \end{pmatrix}.
\end{align}
The system is closed by the boundary conditions 
\begin{alignat}{3}
    \label{eq:MNSFP_BC}
    \Grad_{\Coordinate} \Vec{\Phi} \cdot  \Vec{n} 
    & = \vec{0} \quad 
    && \text{ on } (0,T) \times \partial \CoordinateSpace, \\ 
    \Grad_{\Coordinate} \Velocity \cdot \Vec{n}
    & = \vec{0} \quad 
    && \text{ on } (0,T) \times \Gamma_{N}, \\   
    \Velocity 
    & = \Velocity_D \quad 
    && \text{ on } (0,T) \times \Gamma_D.
\end{alignat}
For the macroscopic time-fractional NSFP system, starting from rest corresponds to the initial conditions $\Velocity(0, \Coordinate) = \vec{0}$ and $\Vec{\Phi}(0, \Coordinate) = (1,0,0,0,0,0,0)\Transpose$ for $\Coordinate \in \Omega$.
Note that only $\HermiteModes{0,0,0}$ contains the ``mass'', cf. \eqref{eq:Hermite_functions_zero_mass}, and due to the chosen initial and boundary conditions, this implies $\HermiteModes{0,0,0} \equiv 1$.

\section{Numerical method}
\label{sec:numerical_scheme}
Applying numerical methods directly to the time-fractional Hookean-type NSPE system is infeasible since it is defined on the Cartesian product of two $d$-dimensional spaces, and it is nonlocal in time due to the time-fractional derivative.
Accepting the concession of not determining the solution of the TFFP equation but only the extra-stress tensor, in the previous section, we applied the Hermite spectral method to the TFFP equation and derived a purely macroscopic model for the extra-stress tensor, hence reducing the problems' complexity to solving a system of seven coupled $d$-dimensional time-fractional PDEs.
For the non-locality, we utilize the kernel compression method introduced in Section \ref{subsec:kernel_compression_method}, approximating the time-fractional derivative of the solution with a sum of $m$ fractional modes $\Vec{\Phi}_k$, $k = 1, \ldots, m$ which solve a system of ODEs.
The resulting macroscopic PDE system reads as follows, whereby from now on, we use $\Grad$ and $\Laplace$ to denote the spatial gradient and Laplace operator, respectively.
\begin{alignat}{3}
    \label{eq:approximate_macroscopic_navier_stokes_fokker_planck_1}
    \PDiff{t} \Vec{\Phi} 
    +  \sum_{k=1}^m\big(
    \left(\Velocity \cdot \nabla \right) \PDiff{t} \Vec{\Phi}_k 
    - \comdiff \Laplace \PDiff{t} \Vec{\Phi}_k 
    - \tensor{A}(\Grad \Velocity) \PDiff{t} \Vec{\Phi}_k 
    \big)
    & = 0
    \quad \text{ in } (0,T) \times \Omega, \\ 
    \label{eq:approximate_macroscopic_navier_stokes_fokker_planck_2}
    \text{for } k = 1, \ldots, m, \quad 
    \PDiff{t} \Vec{\Phi}_k 
    + \tfpoles \Vec{\Phi}_k - \tfweights \Vec{\Phi} 
    & = 0
    \quad \text{ in } (0,T) \times \Omega, \\ 
    \label{eq:approximate_macroscopic_navier_stokes_fokker_planck_3}
    \PDiff{\Time} \Velocity
    + \left( \Velocity \cdot \Grad \right) \Velocity 
    + \Grad \Pressure 
    - \frac{\beta}{\Reynolds} \Laplace \Velocity 
    - \sum_{k=1}^m \Grad \cdot \ExtraStress \left( \PDiff{t} \Vec{\Phi}_k \right)
    & = 0
    \quad \text{ in } (0,T) \times \Omega, \\ 
    \label{eq:approximate_macroscopic_navier_stokes_fokker_planck_4}
    \Grad \cdot \Velocity 
    & = 0,  
    \quad \text{ in } (0,T) \times \Omega, 
\end{alignat}
subject to initial conditions 
\begin{align}
    \label{eq:amnsfp_initial_conditions}
    \Velocity(0,\Coordinate) = \Velocity^0, \quad 
    \Vec{\Phi}(0,\Coordinate) = \Vec{\Phi}^0(\Coordinate), \quad 
    \Vec{\Phi}_k(0,\Coordinate) = \Vec{0}.
\end{align}
Starting from rest, as its initial state, we choose $\Vec{\Phi}^0 \equiv \left(1,0,0,0,0,0,0\right)\Transpose$, and  $\Velocity^0 \equiv \Vec{0}$.
Further, we consider mixed boundary conditions for the NS equations and pure Neumann boundary conditions for the MFP equation 
\begin{alignat}{3}
    \label{eq:amnsfp_boundary_conditions}
    \Velocity & = \Velocity_D \quad && \text{ on } (0,T) \times \Gamma_D, \\ 
    \Grad \Velocity \cdot \Vec{n} & = \vec{0} \quad && \text{ on } (0,T) \times \Gamma_N, \\ 
    \Grad \Vec{\Phi} \cdot \Vec{n} & = \vec{0} \quad && \text{ on } (0,T) \times \partial \Omega.
\end{alignat}

In this section, we describe the temporal and spatial discretization of the PDE system \eqref{eq:approximate_macroscopic_navier_stokes_fokker_planck_1}--\eqref{eq:amnsfp_boundary_conditions}.
Applying a projection method to the NS equations and a standard time integration to the MFP equation, we derive the semi-discrete system, and after expanding the finite element formulation with respect to the finite element bases the fully discrete system.

\subsection{Projection method and time integration}

For the NS equations, we utilize a projection method, as presented in \cite{franco2020high,tomboulides1997numerical}, which is based on the initial work of Chorin \cite{chorin1967numerical,chorin1968numerical}.
Extrapolating the nonlinear velocity term and the extra-stress tensor and considering the divergence of the momentum equation, we first solve a Poisson-type equation for the pressure, and reintroducing the already determined pressure into the momentum equation, we solve a Helmholtz-type equation for the velocity \cite{franco2020high}.
This implicit-explicit time integration method can achieve up to third-order global convergence in time for the velocity \cite{guermond2006overview}.

Subsequently, we solve the MFP equation using the previously determined velocity and update the fractional modes by a linear push forward.
Applying a time integration of order two or higher to the approximate MFP equation, the convergence order in time is at most $1 + \alpha$ for the solution; see, e.g., \cite{khristenko2023solving}.
Consequently, only first-order convergence is achieved for the extra-stress tensor, as it depends on the time-fractional derivative of the MFP solution.

We define the extrapolation operator $\mathcal{E}^{n+1}$ and introduce the backward differentiation formula (BDF), both of order $g \in \N$: 
\begin{equation}
    \label{eq:BDF_and_extrapolation}
    f^{n+1} \approx \mathcal{E}^{n+1}(f) := \sum_{j=1}^g a_j f^{n+1-j}, \qquad 
    \partial_t f \approx \sum_{j=0}^g \frac{b_j}{\Delta t} f^{n+1-j}.
\end{equation}
For $g = 1$, we obtain $a_1 = b_0 = 1$ and $b_1 = -1$, and for $g=2$ the coefficients are given by 
\begin{equation}
    a_0 = 2, \quad 
    a_1 = -1, \quad 
    b_0 = \frac{3}{2}, \quad 
    b_1 = -2, \quad 
    b_2 = \frac{1}{2}.
\end{equation}
We introduce the velocity-correction formulation for calculating the pressure, overcoming the limitation to first-order convergence in time for the velocity; see \cite{guermond2006overview,karniadakis1991high}.
The linear term $\Laplace \Velocity$ is rewritten as 
\begin{equation}
    \label{eq:velocity_correction}
    \nabla (\nabla \cdot \Velocity) - \nabla \times \nabla \times \Velocity, 
\end{equation}
and by setting the first term to zero, we weakly enforce the incompressibility constraint.
Introducing \eqref{eq:BDF_and_extrapolation} and \eqref{eq:velocity_correction} into the momentum equation, we obtain 
\begin{align}
    \label{eq:time_discretization_NS}
    \frac{b_0}{\Delta t} \Velocity^{n+1} 
    & = - \nabla \Pressure^{n+1} 
    \underbrace{ -
        \sum_{j=1}^g \frac{b_j}{\Delta t} \Velocity^{n+1-j}
        - \mathcal{E}^{n+1}\left( 
        \frac{\beta}{\Reynolds} \left( \nabla \times \nabla \times \Velocity \right) 
        + (\Velocity\cdot \nabla) \Velocity 
        - \nabla \cdot \ExtraStress \right)
    }_{=: - \vec{\Xi}^{n+1}}.
 \end{align}
To determine the pressure, we take the divergence of equation \eqref{eq:time_discretization_NS} and obtain a Poisson-type equation 
\begin{align}
    \label{eq:NS_Poisson}
    \Delta \Pressure^{n+1} = 
    - \frac{b_0}{\Delta t} \nabla \cdot \Velocity^{n+1} 
    - \nabla \cdot \vec{\Xi}^{n+1}.
\end{align}
Because of the continuity equation, $\Grad \cdot \Velocity^{n+1} = 0$. The velocities numerically calculated at previous time steps are not divergence-free and, thus, appear in the pressure equation.
The system is closed by the boundary condition 
\begin{alignat}{3}
    \Grad \Pressure ^{n+1} \cdot \Vec{n}
    & = \vec{\Xi}^{n+1} \cdot \Vec{n} - \frac{b_0}{\Delta t} \Velocity_D^{n+1} \cdot \Vec{n}, 
    \quad && \text{on } \Gamma_D, \\ 
    \Grad \Pressure ^{n+1} \cdot \Vec{n}
    & = \vec{\Xi}^{n+1} \cdot \Vec{n}, 
    \quad && \text{on } \Gamma_N.
\end{alignat}
If only Neumann boundary conditions are considered, we utilize a mean-zero condition on the pressure to fully specify the system.
Reintroducing the pressure $\Pressure^{n+1}$ into the momentum equation, we determine $\Velocity^{n+1}$ by solving the remaining Helmholtz-type equation 
\begin{align}
    \label{eq:NS_Helmholtz}
    \frac{b_0}{\Delta t} \Velocity^{n+1} 
    - \frac{\beta}{\Reynolds} \Delta \Velocity^{n+1}
    & = - \nabla \Pressure^{n+1} 
    - \sum_{j=1}^g \frac{b_j}{\Delta t} \Velocity^{n+1-j} 
    + \mathcal{E}^{n+1}\left( - (\Velocity \cdot \nabla) \Velocity + \nabla\cdot \ExtraStress \right), 
\end{align}
subject to mixed Dirichlet and homogeneous Neumann boundary conditions.
We solve the approximate MFP equation based on the current velocity $\Velocity^{n+1}$.
Applying the BDF scheme to the fractional mode equations results in 
\begin{align}
    \sum_{j=0}^g b_j \Vec{\Phi}_k^{n+1-j} 
    + \tfpoles \Delta t \Vec{\Phi}_k^{n+1} - \tfweights \Delta t \Vec{\Phi}^{n+1} 
    = 0,
\end{align}
for $k=1, \ldots, m$, which we can explicitly solve for $\Vec{\Phi}_k^{n+1}$ 
\begin{align}
    \label{eq:BDF_Phi_k_n+1}
    \Vec{\Phi}_k^{n+1} 
    = \frac{
        - \sum_{j=1}^g b_j \Vec{\Phi}_k^{n+1-j} 
        + \tfweights \Delta t \Vec{\Phi}^{n+1}
    }{b_0 + \tfpoles \Delta t}.
\end{align}
For the MFP equation, we insert Equation \eqref{eq:approximate_macroscopic_navier_stokes_fokker_planck_2} to replace the time derivatives of the fractional modes and obtain 
\begin{align}
    \PDiff{t} \Vec{\Phi} 
    + \sum_{k=1}^m \big(
    \left(\Velocity \cdot \nabla \right)  
    - \comdiff \Laplace 
    - \tensor{A}(\Grad \Velocity)\big) \left(- \tfpoles \Vec{\Phi}_k + \tfweights \Vec{\Phi}\right)
    = 0.
\end{align}
Applying the BDF scheme and replacing the occurring $\Vec{\Phi}_k^{n+1}$ terms inserting \eqref{eq:BDF_Phi_k_n+1}, we obtain 
\begin{multline}
    \label{eq:FP_discrete}
    \Big( b_0 
    + \eta
        \left( (\Velocity^{n+1} \cdot \nabla) 
        - \comdiff \Laplace 
        - \tensor{A}(\Grad \Velocity^{n+1})\right) 
    \Big) \Vec{\Phi}^{n+1} \\  
    = - \sum_{j=1}^g b_j \left(
    \Vec{\Phi}^{n+1-j} 
    + \sum_{k=1}^m \eta_k
        \left( (\Velocity^{n+1} \cdot \nabla)   
        - \comdiff \Laplace 
        - \tensor{A}(\Grad \Velocity^{n+1}) \right)\Vec{\Phi}_k^{n+1-j}
    \right), 
\end{multline}
where the coefficients $\eta$
and $\eta_k$ are defined by 
\begin{align}
    \eta := \sum_{k=1}^m \frac{b_0 w_k \Delta t }{b_0 + \lambda_k \Delta t}, \qquad
    \eta_k := \frac{ \tfpoles \Delta t}{b_0 + \tfpoles \Delta t}.
\end{align}
We close the system with homogeneous Neumann boundary conditions and solve for $\Vec{\Phi}^{n+1}$.
The update of the fractional modes $\Vec{\Phi}_k^{n+1}$ is only a linear push forward using \eqref{eq:BDF_Phi_k_n+1} and does not require solving a system of equations.

\subsection{Finite element approximation}

Using the projection method for the NS system and a standard time discretization of the approximate MFP system, we have to solve a Poisson-type equation \eqref{eq:NS_Poisson}, a Helmholtz-type equation \eqref{eq:NS_Helmholtz} and the advection-diffusion-reaction system \eqref{eq:FP_discrete}.
Note that, due to the projection method, we never solve a saddle-point problem, and the continuity equation of the NS system is imposed weakly. 
The corresponding finite element formulations are described below.

The spatial domain $\Omega \subset \R^d$, $d=2,3$ is discretized using an unstructured mesh of hexahedral (quadrilateral in 2D) elements $\mathcal{T}_h$, and we replace the infinite-dimensional spaces in the variational setting with second-order continuous finite element spaces in the physical coordinate: 
\begin{alignat}{4}
    S
    & = \{ s \in H^1(\CoordinateSpace) \ 
    && | \ s_{|K} \in Q_2(K) \ 
    &&& \forall K \in \mathcal{T}_h \}, \\
    V 
    & = \{ \Vec{v} \in H^1(\CoordinateSpace)^d \ 
    && | \ \Vec{v}_{|K} \in Q_2(K)^d \ 
    &&& \forall K \in \mathcal{T}_h \}, \\ 
    M 
    & = \{ \Vec{m} \in H^1(\CoordinateSpace)^{3d-2} \ 
    && | \ \Vec{m}_{|K} \in Q_2(K)^{3d-2} \ 
    &&& \forall K \in \mathcal{T}_h \},
\end{alignat}
for the pressure, the velocity, and the modes of the MFP system, respectively.
For Equation \eqref{eq:NS_Poisson} of Poisson-type, the finite element formulation reads as follows: Find $\Pressure^{n+1} \in S$, such that $\forall s \in S$, 
\begin{align}
    \label{eq:discrete:pressure}
    - \left( \Grad \Pressure^{n+1}, \Grad s \right)_{\CoordinateSpace}
    &= \left(\nabla \cdot \vec{\Xi}^{n+1}, s \right)_{\CoordinateSpace} 
    - \frac{b_0}{\Delta t} \left( \Velocity_D^{n+1} \cdot \Vec{n}, s\right)_{\Gamma_D} 
    + \left( \vec{\Xi}^{n+1} \cdot \Vec{n}, s\right)_{\partial \Omega}.
\end{align}
The curl operator occurring in $\vec{\Xi}^{n+1}$ is computed as the average of all local neighboring elements projected onto the finite element space. 
In the case of pure Neumann boundary conditions, i.e., $\Gamma_D = \emptyset$, the system is closed by a zero mean condition on the pressure 
\begin{equation}
    \label{eq:mean_zero_pressure_condition}
    \int_{\CoordinateSpace} \Pressure^{n+1} \,\Diff \Coordinate = 0.
\end{equation}
For the Helmholtz-type Equation \eqref{eq:NS_Helmholtz} with mixed boundary conditions, we obtain: Find $\Velocity^{n+1} \in \{\Vec{v} \in V, \Vec{v}|_{\Gamma_D} = \Velocity_D^{n+1}\},$ such that for all $\Vec{v} \in \{\Vec{v} \in V, \Vec{v}|_{\Gamma_D}  = \Vec{0}\}$
\begin{align}
    \frac{b_0}{\Delta t} \left(\Velocity^{n+1}, \Vec{v} \right)_{\CoordinateSpace} 
    - \frac{\beta}{\Reynolds} \left(\Grad \Velocity^{n+1}, \Grad \Vec{v}\right)_{\CoordinateSpace} 
    =  
    - \left(\nabla \Pressure^{n+1},\Vec{v}\right)_{\CoordinateSpace} 
    - \left(\vec{\Xi}^{n+1}, \Vec{v}\right)_{\CoordinateSpace}.
\end{align}
Since the pressure term is not partially integrated, a homogeneous Neumann boundary condition for the velocity is imposed naturally on $\Gamma_N$.
Finally, the finite element formulation for the MFP equation reads as find $\Vec{\Phi}$ such that for all $\Vec{m} \in M$
\begin{multline}
    \label{eq:discrete:modes}
    b_0 \left(\Vec{\Phi}^{n+1}, \Vec{m}\right)_{\CoordinateSpace}
    + \eta
    \Bigl(
        \left((\Velocity^{n+1} \cdot \nabla) \Vec{\Phi}^{n+1}, \Vec{m} \right)_\Omega 
        +  \comdiff \left(\Grad \Vec{\Phi}^{n+1}, \Grad \Vec{m}\right)_\Omega  
        - \left(\tensor{A}(\Grad \Velocity^{n+1}) \Vec{\Phi}^{n+1}, \Vec{m}\right)_\Omega 
    \Bigr) \\
    = - \sum_{j=1}^g b_j 
    \Biggl(
        \sum_{k=1}^m \eta_k
        \Bigl(
            \big(
                (\Velocity^{n+1} \cdot \nabla) \Vec{\Phi}_k^{n+1-j}, \Vec{m}
            \big)_\Omega
            + \comdiff \bigl( \Grad \Vec{\Phi}_k^{n+1-j}, \Grad \Vec{m}\bigr)_\Omega
            - \big(
                \tensor{A}(\Grad \Velocity^{n+1})\Vec{\Phi}_k^{n+1-j}, \Vec{m}
            \big)_\Omega 
        \Bigr) \\
        + \bigl(\Vec{\Phi}^{n+1-j}, \Vec{m}\bigr)_\Omega
    \Biggr).
\end{multline} 
Again, the homogeneous Neumann boundary condition is imposed naturally on $\partial \CoordinateSpace$.

Now, the linear algebraic systems are obtained by expanding $\Velocity,\Pressure,\Vec{\Phi}, \Vec{\Phi}_1, \ldots, \Vec{\Phi}_m$ and the test functions with respect to the bases of their respective finite element spaces in Equations \eqref{eq:discrete:pressure}--\eqref{eq:discrete:modes}.

\section{Numerical results}
\label{sec:numerical_experiments}
Considering sufficiently regular scenarios, we firstly investigate the convergence order of our numerical scheme, showing the optimal convergence for the fractional FP equation against an analytical solution and studying the performance of the fully coupled system for both the fractional and the integer-order case using numerical reference solutions.
Secondly, we illustrate the influence of the order of the fractional derivative on the flow of the dilute polymeric fluid based on the well-known flow around a cylinder example. Thirdly, the influence on turbulence and flow stabilizing properties of the polymer molecules are further examined for geometries replicating an internally corroded pipeline and a pipeline strain relief.

The numerical scheme can be easily implemented in any modern finite element framework. We utilize MFEM \cite{mfem}, an open-source C++ finite element library, inheriting its scalability for the MPI-parallel implementation running on up to 256 cores on a local computation cluster with two AMD EPYC 9754 128-core processors.

\subsection{Convergence studies}

To establish the convergence order of the numerical scheme for the time-fractional FP equation, we consider a decoupled setting with $\Velocity \equiv \Vec{0}$ over the unit square $\Omega = (0,1)^2$, for which we can derive an analytical solution.
Considering an initial condition $\left(\Phi_\text{ana}\right)_i(0, \Coordinate) = \cos(2 \pi x_1) \cos(4 \pi x_2)$, and homogeneous Neumann boundary conditions, the analytical solution is given as 
\begin{equation}
    \left(\Phi_\text{ana}\right)_i(t, \Coordinate) 
    = \cos(2 \pi x_1) \cos(4 \pi x_2) E_\alpha \left(-\left(\frac{1}{2 \Deborah} + 20 \comdiff \pi^2 \right) t^\alpha \right), 
\end{equation}
where $E_\alpha$ denotes the single-parameter Mittag--Leffler function, see, e.g., \cite{kexue2011laplace}.
For the convergence study, we consider $\comdiff = 10^{-2}$, $\Deborah = 0.5$, $T=1$.
Since this simulation starts with a non-trivial initial condition, we calculate the first step for the BDF2 scheme using an SDIRK2-method, as the BDF2 with a BDF1 start is not L-stable \cite{nishikawa2019large}.
The error decay for the kernel compression method was derived for a weighted Bochner space \cite{khristenko2023solving}, whereby the norm is defined as
\begin{equation}
    \label{eq:weighted_Bochner_space_norm}
    \|u\|_\alpha 
    := \left(\int_0^1 \|u(t)\|_{L^2(\Omega)}^2 t^{2(1-\alpha)} \Diff t \right)^{\frac{1}{2}}.
\end{equation}
We display the $\|\cdot\|_\alpha$-norm of the difference between the numerical and the analytical solution for decreasing time step sizes and various number of modes $m$ in Figure \ref{fig:error_decay}(a).
\begin{figure}
    \centering
    \includegraphics[width=\textwidth, trim=0mm 2mm 0mm 2mm, clip]{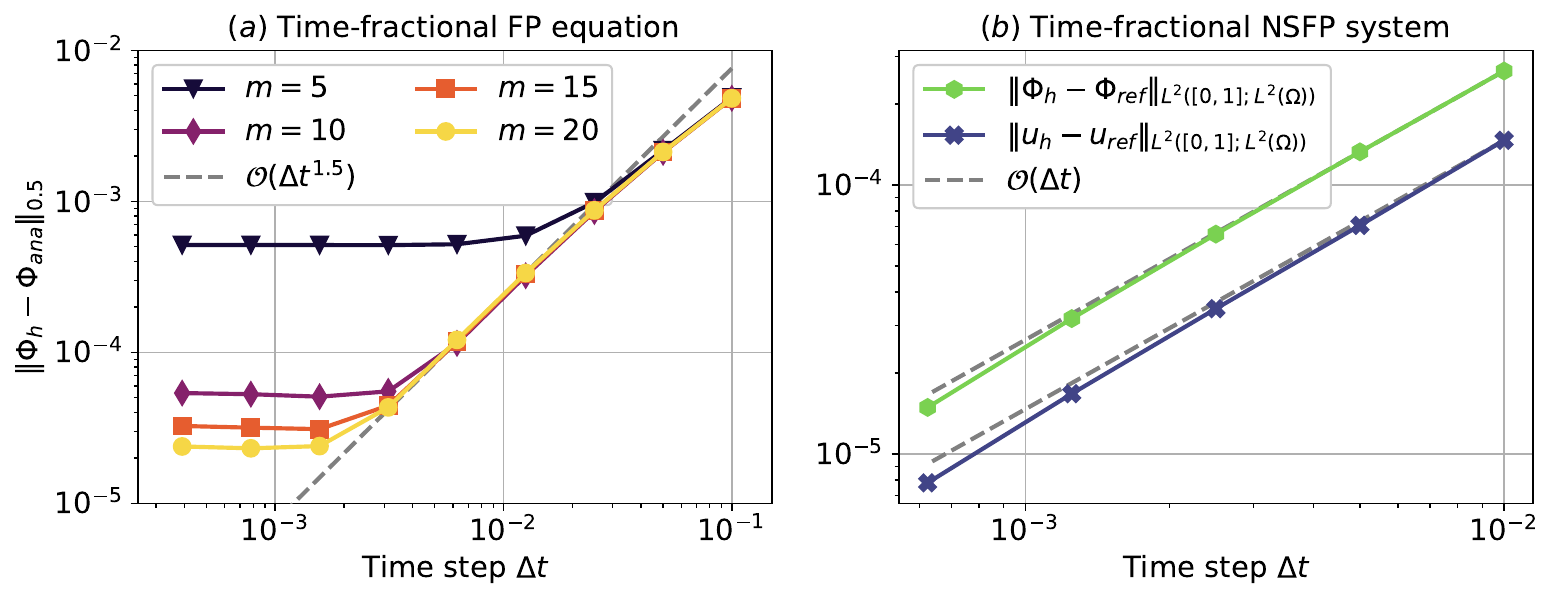}
    \caption{Error decay of the numerical solution (a) in comparison to an analytical solution for the TFFP equation with $\alpha = 0.5$ and (b) in comparison to a numerical reference solution obtained at a finer resolution for the fully coupled time-fractional NSFP system.}
    \label{fig:error_decay}
\end{figure}
The error decays asymptotically with order $\mathcal{O}\left(\Delta t^{1+\alpha}\right)$, whereby the plateau is dependent on the choice of $m$, and thus, due to the accuracy of the approximation of the Riemann--Liouville integral kernel.

For the fully coupled time-fractional NSFP system, no analytical solutions are available, and thus, we rely on the comparison with a numerical reference solution.
Again, we consider $T=1$, $\Omega = (0,1)^2$ and apply no-slip boundary conditions at the top and bottom and periodic boundary conditions at the sides of the domain for the velocity and homogeneous Neumann boundary conditions for the MFP equations.
The flow is driven by an external force $\vec{f}$, starting constant as zero, thus avoiding the necessity of computing an initial step:
\begin{equation}
    \vec{f}(\Time, \Coordinate) 
    = \sin^2 \!\left(\frac{\pi}{2} \min\big\{\max\{0,\, 4\Time - 1\},\, 1\big\}\right) \begin{pmatrix} 1, & 0 \end{pmatrix}\Transpose .
\end{equation}
We set the material parameters as $\beta / \Reynolds = 1$, $\epsilon = 1$, $\Deborah = 0.5$, and $(1-\beta) / \Reynolds = 0.5$ and the reference solution is calculated using a time-step size $\Delta\Time = 7.8125 \cdot 10^{-5}$.
We observe a second-order error decay for the integer-order system and first-order convergence for the fractional system with $\alpha = 0.5$; see Figure \ref{fig:error_decay}(b), which is expected as the extra-stress tensor depends on the time-fractional derivative of the solution. 

\subsection{Flow around a cylinder}

The well-known flow around a cylinder benchmark, see, e.g., \cite{schafer1996benchmark}, is chosen to investigate the influence of the order of the time-fractional derivative on polymeric fluid flow and the drag-reducing effect of the polymer molecules.

We consider the dynamic simulation of the flow through a pipe with a circular cutout $\Omega = (0,2.2) \times (0,0.41) \backslash S$, whereby the obstacle $S = B_{0.05}(0.2,0.2)$ is placed slightly off the center line to break the symmetry of the scenario. 
The mesh comprises 13\,000 rectangular elements, and the time step size is set to $\Delta\Time = 10^{-4}$.  
We consider no-slip boundary conditions at the domain's obstacle, top, and bottom.
At the outflow, we impose homogeneous Neumann boundary conditions on the velocity, and we impose a quadratic inflow profile on the left side of the domain:
\begin{align}
    \label{eq:KVS_inlet_velocity}
    \Velocity_\text{in}(\Time, \Coordinate) =
    \sin^2\!\left(\frac{\pi}{2} \min\{\Time, 1\}\right) \begin{pmatrix} \dfrac{6 x_2 (0.41 - x_2)}{0.41^2}, & 0 \end{pmatrix}\Transpose .
\end{align}
The non-standard time dependency of the inflow profile was introduced to evade the difficulties of time-fractional initial value problems for $\alpha \neq 1$. For this benchmark, the viscosity of the solvent fluid $\beta / \Reynolds = 10^{-3}$, the viscosity contribution of the polymer molecules $(1-\beta) / \Reynolds = 0.5$, the center-of-mass diffusion $\comdiff = 1$, and the Deborah number $\Deborah = 0.5$. The parameters are chosen to obtain the standard periodic Karman-vortex-street for the pure solvent fluid and laminar flow for the dilute polymeric fluid with $\alpha = 1$, respectively.

We consider the same setting as for the dilute polymeric fluid but with $\alpha = 0.5, 0.8$. Both fractional cases result in turbulent flow, whereby the velocity magnitude for $\alpha = 0.5$ is increased compared to $\alpha = 0.8$. The influence of the polymer molecules depending on the order of the time-fractional derivative is quantified in terms of the magnitude of the polymer-induced force in Figure \ref{fig:Influence of fractional order}.
\begin{figure}
    \centering
    \includegraphics[width=\textwidth, trim=0mm 3mm 0mm 2mm, clip]{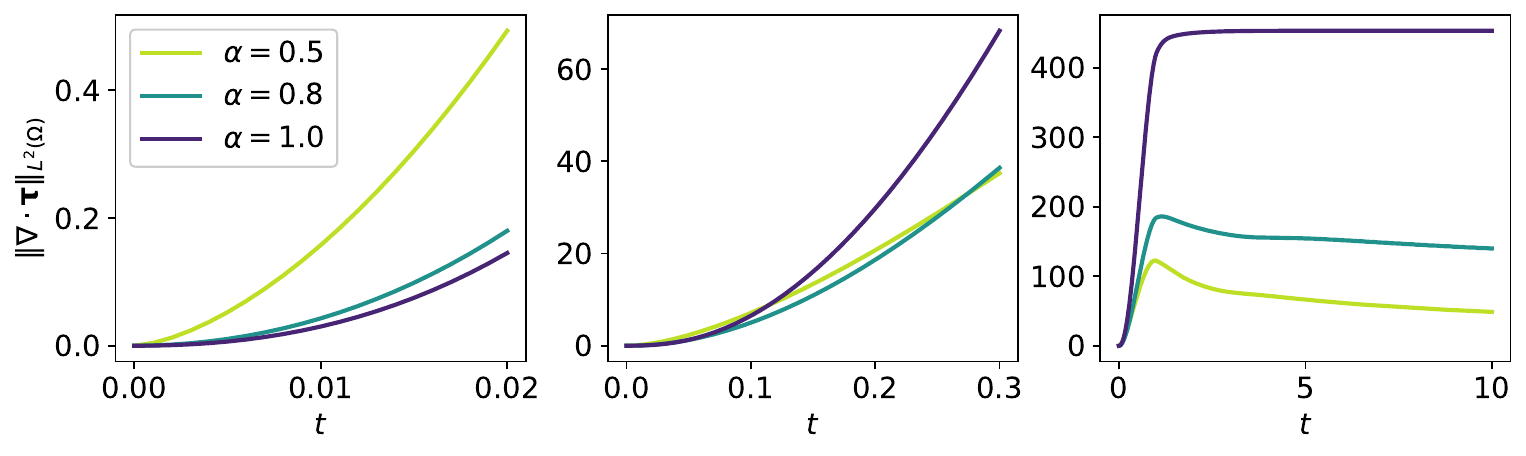}
    \caption{Influence of the fractional order on the magnitude of the polymer-induced force $\|\Grad \cdot \ExtraStress\|_{L^2(\CoordinateSpace)}$.}
    \label{fig:Influence of fractional order}
\end{figure}
\begin{figure}
    \centering
    \includegraphics[width=\textwidth, trim=0mm 2mm 0mm 2mm, clip]{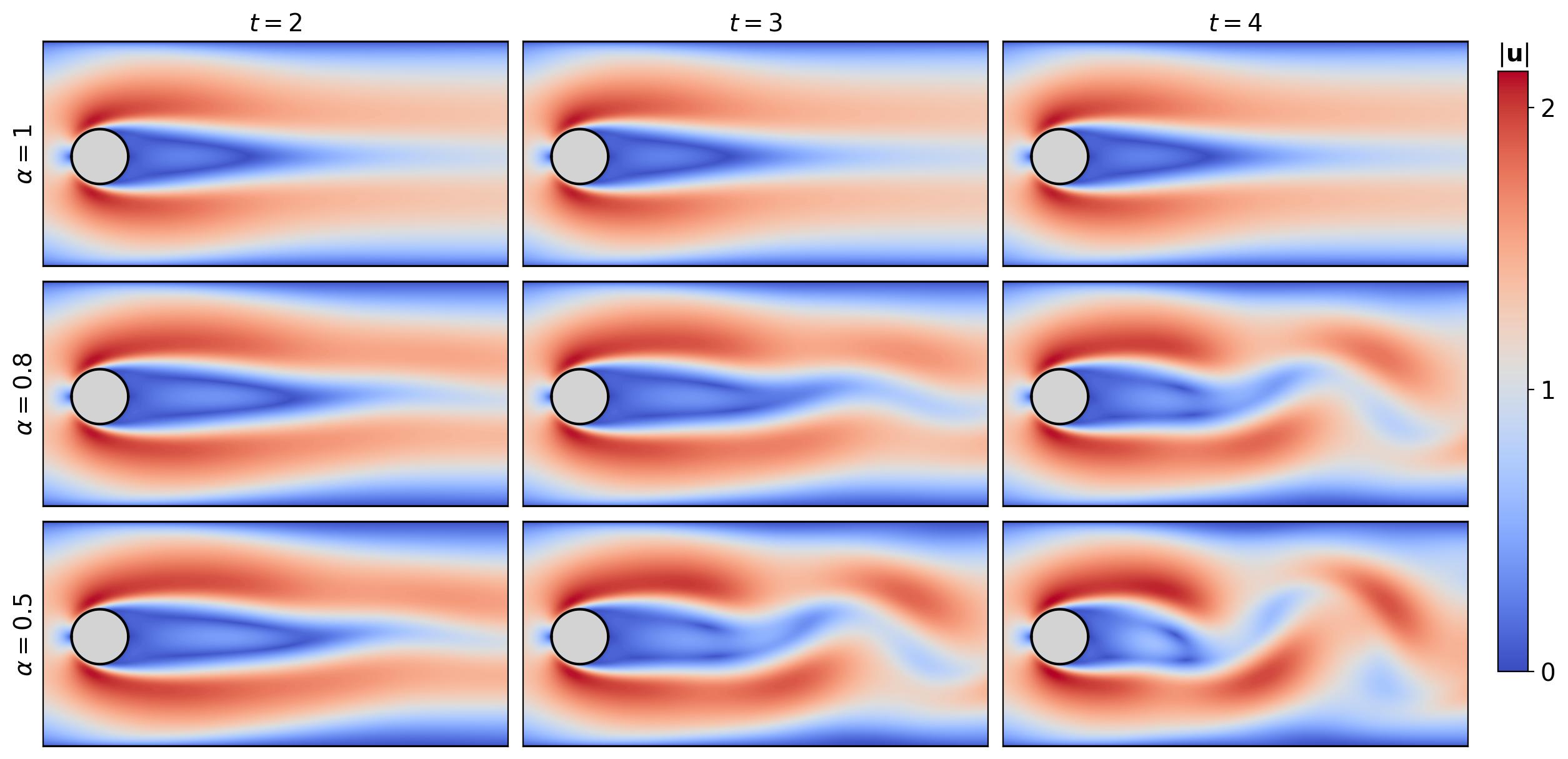}
    \caption{Comparison of the velocity magnitude of dilute polymeric fluid flow for $\alpha = 0.5, 0.8, 1$ at $t=2,3,4$.}
    \label{fig:case4_flow_profiles}
\end{figure}
The extra-stress tensors for smaller $\alpha$ show a faster initial response, but throughout the simulation, the magnitude of the extra-stress tensor remains lower than those of larger values of $\alpha$, resulting in a smaller polymer-induced force. This evolution over time is typically for the single-parameter Mittag--Leffler function and, thus, it is inherent to time-fractional PDEs, e.g., the time-fractional Cahn--Hilliard equations \cite{fritz2022time,khristenko2023solving}. The flow profiles for $\alpha = 0.5, 0.8, 1$ at $t = 2,3,4$ are visualized in Figure \ref{fig:case4_flow_profiles}.
The integer-order case remains consistently laminar, while both time-fractional cases transition from laminar toward turbulent flow. For $\alpha=0.5$, the development of the turbulence is faster and more pronounced than for $\alpha = 0.8$, matching the evolution of the polymer-induced force.

\subsection{Rough wall channel flow}
\label{subsec:rough_wall_channel_flow}

Natural gas and liquid pipelines are subject to internal corrosion \cite{fessler2008pipeline} due to insufficient cathodic protection and coating conditions \cite{vanaei2017review}. 
In this numerical case study, the corroded interior pipeline surface is modeled by a domain $\Omega \subset (0,2.2) \times (0,0.41)$, with 50 irregular quadrilateral cutouts at the top and bottom of the enclosing rectangle. The height at each corner and length of the cutouts are uniformly distributed between $2.5 \cdot 10^{-3}$ and $7.5 \cdot 10^{-3}$ and $4.35 \cdot 10^{-2}$ and $4.45 \cdot 10^{-2}$, respectively. The resulting geometry is visualized in Figure \ref{fig:Omega_Method_NS}. For the numerical simulation, we consider a mesh of $123\,200$ quadrilateral elements and $\Delta t = 10^{-4}$. 

At the domain's top and bottom, we apply no-slip boundary conditions for the fluid and homogeneous Neumann boundary conditions for the MFP equation. At the sides, periodic boundary conditions are prescribed.
A constant external force $\vec{f} \equiv (1,\, 0)\Transpose$ drives the flow in $x_1$ direction. We overcome the computationally expensive build-up times of such scenarios by prescribing an initial velocity field $\Velocity^0(\Coordinate) = (x_2^2(0.41-x_2)^2 / 0.205^2,\, 0)\Transpose$. 

To identify turbulence, we employ the $\omega$-method \cite{liu2016new}, as it inherently overcomes the drawback of choosing suitable threshold values of numerous other vortex identification methods, such as the commonly used $Q$-criterion and $\lambda_2$-criterion.
While the definitions of these criteria are mathematically different, in practice, their results are nearly identical \cite{xi2019turbulent}.
The $\omega$-method is defined as 
\begin{equation}
    \omega = \frac{b}{a + b + \epsilon},
    \quad a = \frac{1}{4} \left\| \Grad \Velocity + \Grad \Velocity \Transpose \right\|_F^2, 
    \quad b = \frac{1}{4} \left\| \Grad \Velocity - \Grad \Velocity \Transpose \right\|_F^2.
\end{equation}
where $\|\cdot\|_F$ is the Frobenius norm, and $\epsilon(t)$ is introduced to avoid non-physical noise. 
As outlined in \cite{dong2018determination}, we empirically choose 
\begin{equation}
    \epsilon(t) = 10^{-3} \max_{\Coordinate \in \CoordinateSpace} \big(b(\Time, \Coordinate) - a(\Time, \Coordinate)\big).
\end{equation}

\begin{figure}[p]
    \centering
    \includegraphics[width=\textwidth, trim=0mm 3mm 0mm 3mm, clip]{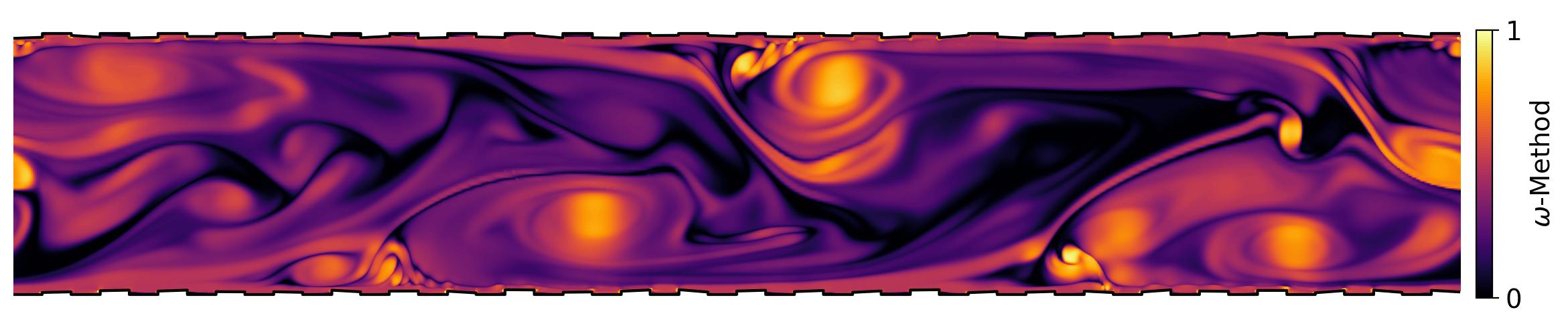}
    \caption{Visualization of the $\omega$-method for the pure solvent fluid at $t=20$.}
    \label{fig:Omega_Method_NS}
\end{figure}


\begin{figure}
    \centering
    \begin{subfigure}[b]{0.49\textwidth}
        \centering
        \includegraphics[width=\textwidth, trim=0mm 5mm 0mm 5mm, clip]{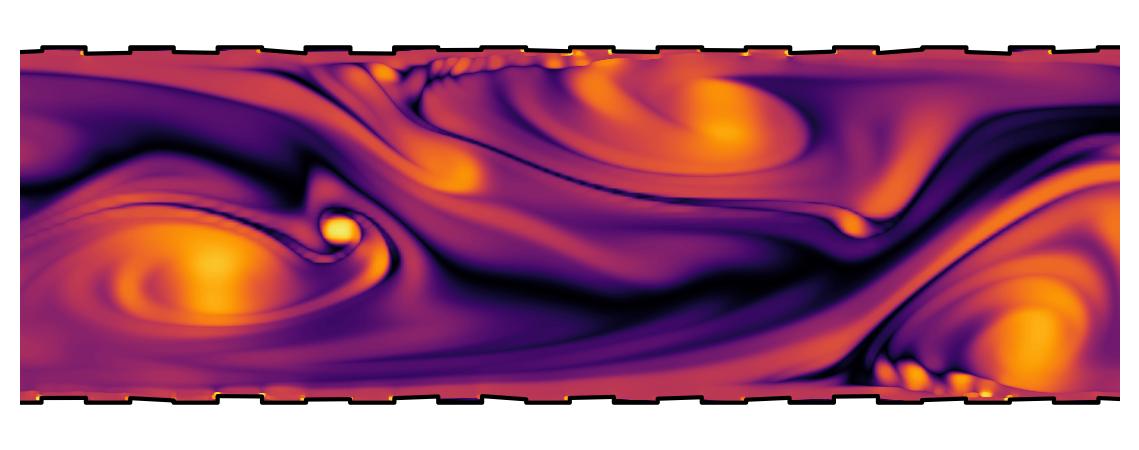}
        \caption{$\comdiff = 1$}
    \end{subfigure} 
    \hfill 
    \begin{subfigure}[b]{0.49\textwidth}
        \centering
        \includegraphics[width=\textwidth, trim=0mm 5mm 0mm 5mm, clip]{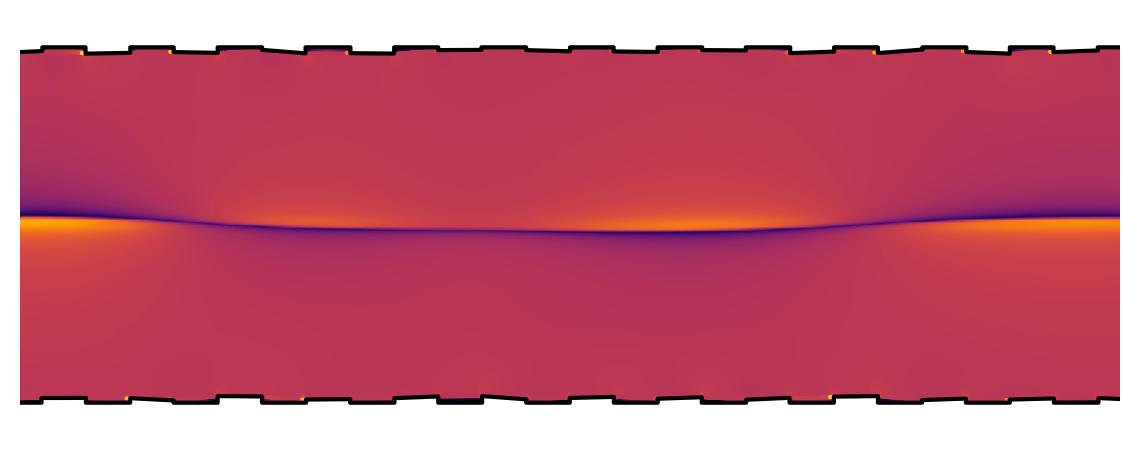}
        \caption{$\comdiff = 10^{-1}$}
    \end{subfigure}
    \medskip

    \begin{subfigure}[b]{0.49\textwidth}
        \centering
        \includegraphics[width=\textwidth, trim=0mm 5mm 0mm 5mm, clip]{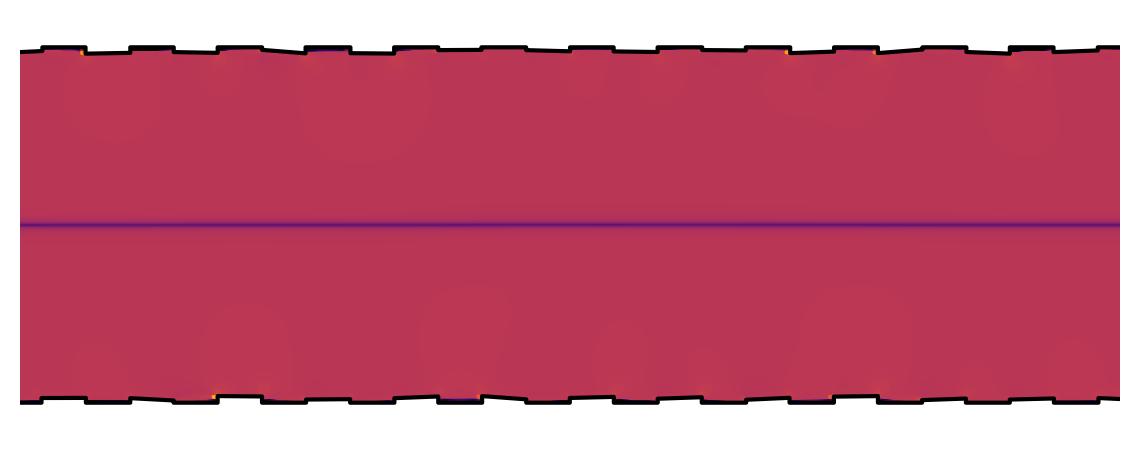}
        \caption{$\comdiff = 10^{-2}$}
    \end{subfigure}
    \hfill 
    \begin{subfigure}[b]{0.49\textwidth}
        \centering
        \includegraphics[width=\textwidth, trim=0mm 5mm 0mm 5mm, clip]{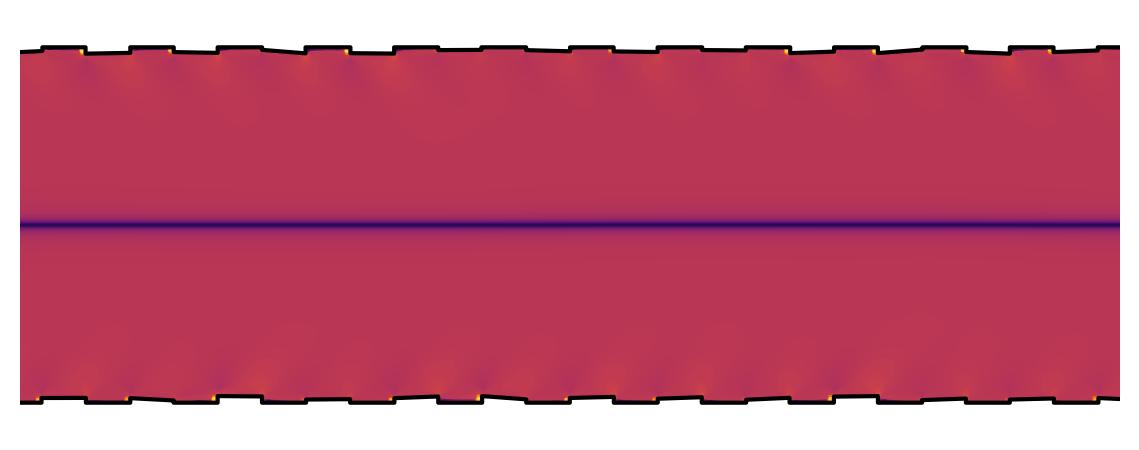}
        \caption{$\comdiff = 10^{-3}$}
    \end{subfigure}
    \caption{Visualization of $\omega$ on $[0.55,1.65] \times [0,0.41]$, for $(1-\beta) / \Reynolds = 10^{-3}$, $\Deborah = 1$, and various $\comdiff$, at $t=20$.}
    \label{fig:Omega_various_comdiff}
\end{figure}


\begin{figure}
    \centering
    \begin{subfigure}[b]{\textwidth}
        \centering
        \includegraphics[width=\textwidth, trim=0mm 2mm 0mm 2mm, clip]{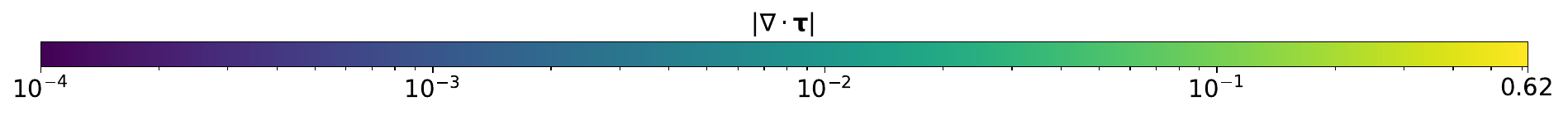}
    \end{subfigure}
    \begin{subfigure}[b]{0.49\textwidth}
        \centering
        \includegraphics[width=\textwidth, trim=0mm 5mm 0mm 5mm, clip]{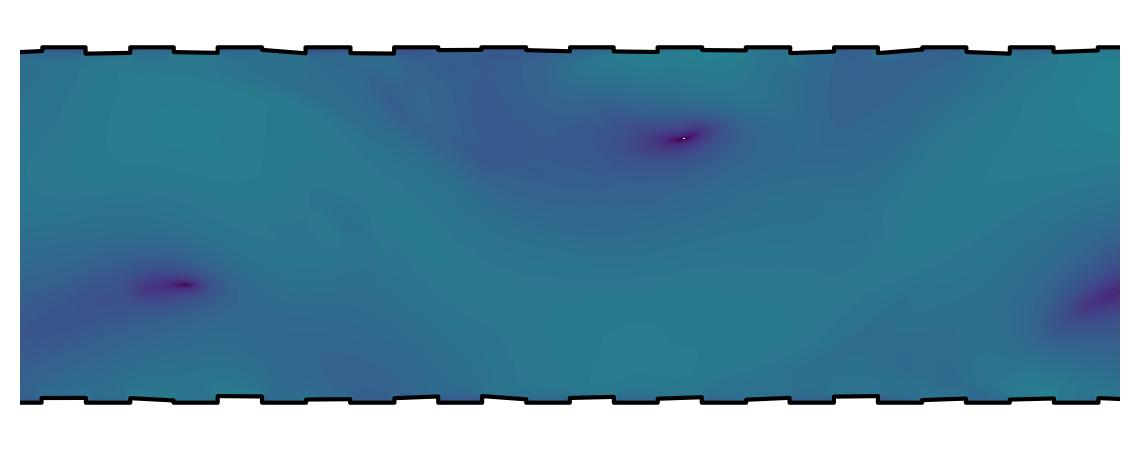}
        \caption{$\comdiff = 1$}
    \end{subfigure} 
    \hfill 
    \begin{subfigure}[b]{0.49\textwidth}
        \centering
        \includegraphics[width=\textwidth, trim=0mm 5mm 0mm 5mm, clip]{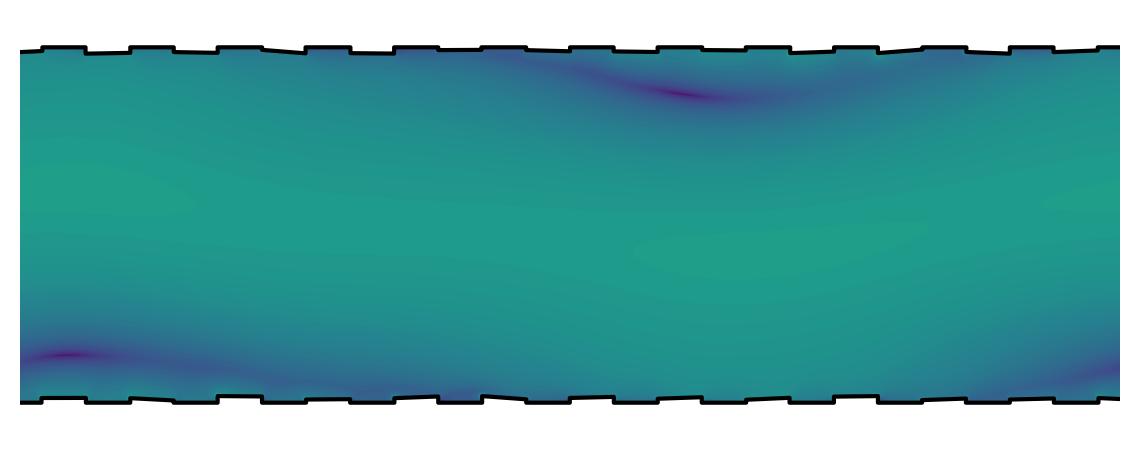}
        \caption{$\comdiff = 10^{-1}$}
    \end{subfigure}
    \medskip

    \begin{subfigure}[b]{0.49\textwidth}
        \centering
        \includegraphics[width=\textwidth, trim=0mm 5mm 0mm 5mm, clip]{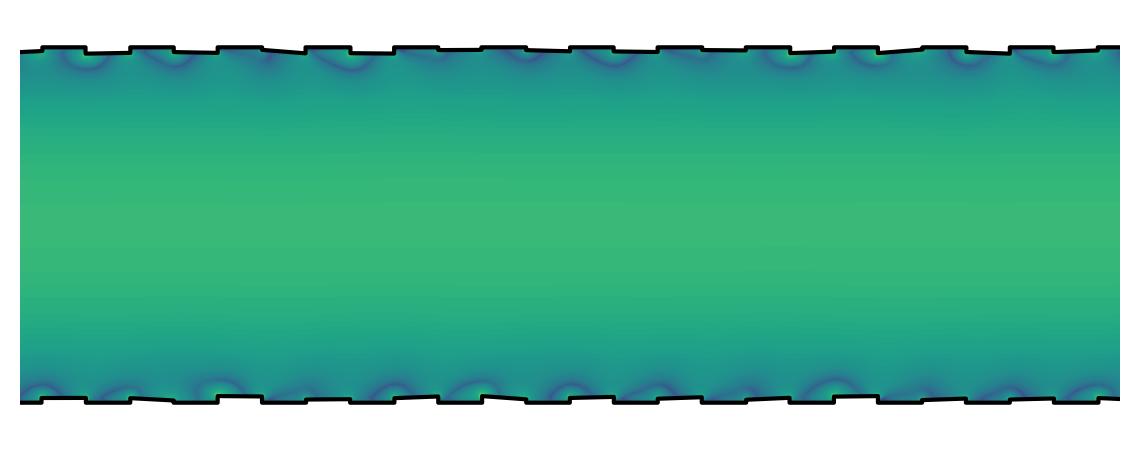}
        \caption{$\comdiff = 10^{-2}$}
    \end{subfigure}
    \hfill 
    \begin{subfigure}[b]{0.49\textwidth}
        \centering
        \includegraphics[width=\textwidth, trim=0mm 5mm 0mm 5mm, clip]{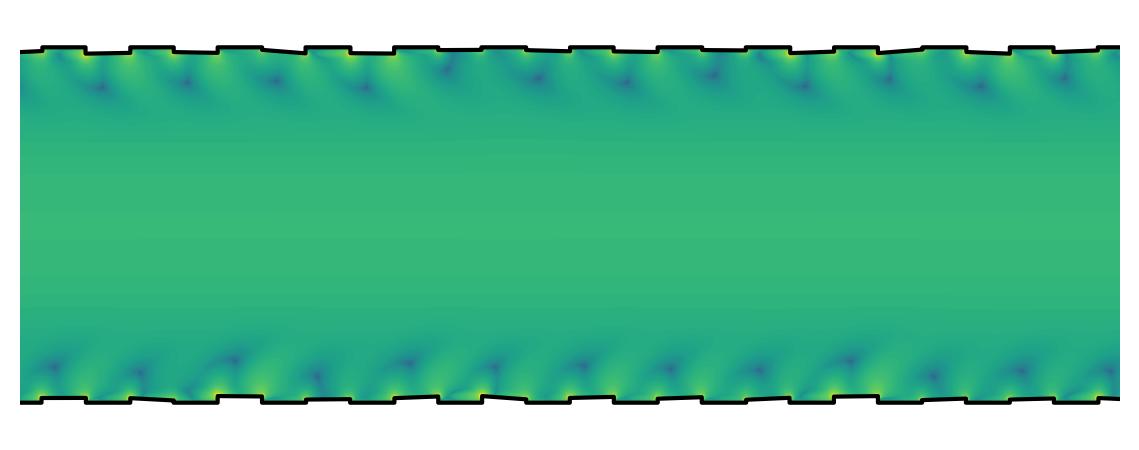}
        \caption{$\comdiff = 10^{-3}$}
    \end{subfigure}
    \caption{Polymeric force magnitude $\Abs{\nabla \cdot \ExtraStress}$ on $[0.55,1.65] \times [0,0.41]$, for various $\comdiff$, $(1-\beta) / \Reynolds = 10^{-3}$, and $\Deborah = 1$, at $t=20$.}
    \label{fig:Forces_various_comdiff}
\end{figure}

As a baseline for the dilute polymeric fluid mixtures, we consider the pure solvent fluid with $\beta /  \Reynolds = 3 \cdot 10^{-5}$. Induced by the irregular geometry of the channel walls, starting at $t=8$, the flow becomes turbulent and remains in this state until the end of the simulation. The fully developed turbulence, at $t=20$, is visualized in terms of $\omega$ in Figure \ref{fig:Omega_Method_NS}. Large values of $\omega$ (bright colors) indicate regions where the flow is mostly vorticity-dominated, while $\omega$ near zero (dark colors) implies strain-dominated flow. We perform simulations of the same scenario across a wide parameter range for dilute polymeric fluids  \cite{DIMITROPOULOS1998433,housiadas2005direct,sousa2011effect,wapperom2000backward,graham2014drag}. Varying the center-of-mass diffusivity $\comdiff$, the viscosity contribution of the added polymer molecules $(1-\beta) / \Reynolds$, and the Deborah number $\Deborah$, we find three qualitatively distinct flow regimes: stable laminar flow, a transition to sustained turbulence as for the pure solvent fluid, and an intermediate regime where turbulence initially develops but subsequently is damped to a laminar flow state. 

The center-of-mass diffusivity is crucial not only in physical aspects but also with respect to the stability of direct numerical simulation of dilute polymeric fluid flow.
To ensure numerical stability, often larger than realistic values are used \cite{graham2014drag}, artificial diffusion terms are added \cite{kim2007effects,xi2019turbulent}, or stabilizing techniques such as upwinding schemes \cite{samanta2013elasto} and slope-limiters \cite{dallas2010strong,shekar2019critical} are employed, thus, introducing numerical diffusion, cf. \cite{dubief2023elasto}. 
While for some mixtures, the center-of-mass diffusivity is as low as $10^{-8}$ \cite{bhave1991kinetic} and thus the diffusion term can be omitted in the model \cite{GRIEBEL201441}, this poses its own numerical challenge, which is outside the scope of this study. 
In the rough wall channel flow scenario, our numerical scheme has shown robustness at $\comdiff \approx 2 /\Reynolds$, corresponding to the magnitude of artificial diffusivity used in other studies, see \cite{graham2014drag} and the references therein.  

The results for $\Deborah = 1$, $\beta / \Reynolds = 3 \cdot 10^{-5}$, $(1-\beta)/\Reynolds = 10^{-3}$, and $\comdiff = 1, 10^{-1}, 10^{-2}, 10^{-3}$, in terms of $\omega$ and $\Abs{\nabla \cdot \ExtraStress}$ at time $t=20$ are displayed in Figures \ref{fig:Omega_various_comdiff} and \ref{fig:Forces_various_comdiff}, respectively.
We observe that for the lower two diffusion values, the flow profile remains laminar, while for $\comdiff=10^{-1}$, the flow initially becomes turbulent before reverting back to a laminar state, and for $\comdiff = 1$, the flow remains turbulent.
The larger the center-of-mass diffusivity, the smoother becomes the extra-stress tensor, and thus, the smaller the polymer-induced force, $\Div \ExtraStress$, as can be seen in Figure \ref{fig:Forces_various_comdiff}. 
Note that the maximal force increases from $8.3 \cdot 10^{-3}$ to 0.62 as the center-of-mass diffusivity decreases from 1 to $10^{-3}$, despite $\Deborah$ and $(1-\beta) / \Reynolds$ being unchanged.   
The force acts opposite the flow direction, resulting in a force profile that reflects the characteristics of the flow field. 
As the center-of-mass diffusivity decreases, the force increases around the geometric features of the domain. 

\begin{figure}
    \centering
    \begin{subfigure}[b]{0.325\textwidth} \centering $t = 9$ \end{subfigure} \hfill 
    \begin{subfigure}[b]{0.325\textwidth} \centering $t = 12$ \end{subfigure} \hfill 
    \begin{subfigure}[b]{0.325\textwidth} \centering $t = 15$ \end{subfigure}  
    \medskip

    \begin{subfigure}[b]{\textwidth}
        \centering
        \includegraphics[width=0.325\textwidth, trim=0mm 5mm 0mm 5mm, clip]{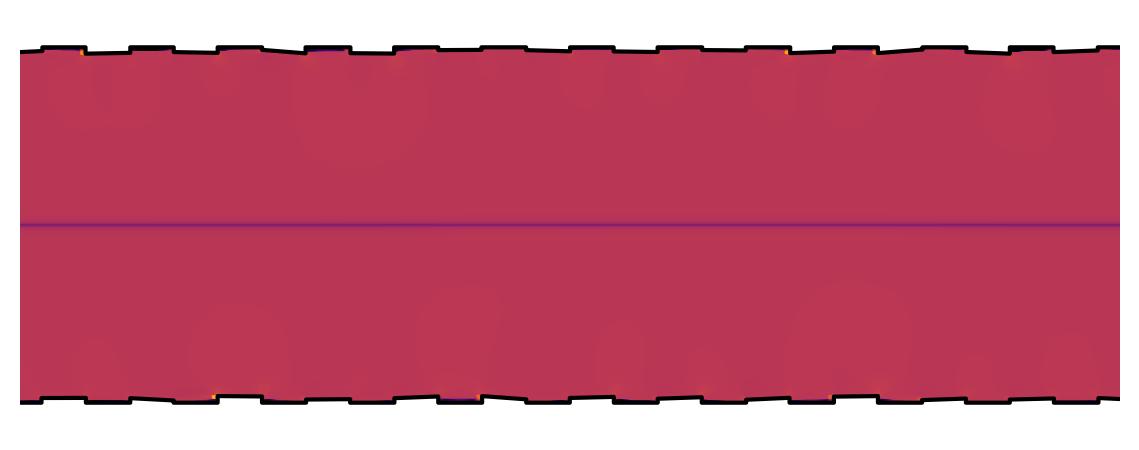}
        \hfill 
        \includegraphics[width=0.325\textwidth, trim=0mm 5mm 0mm 5mm, clip]{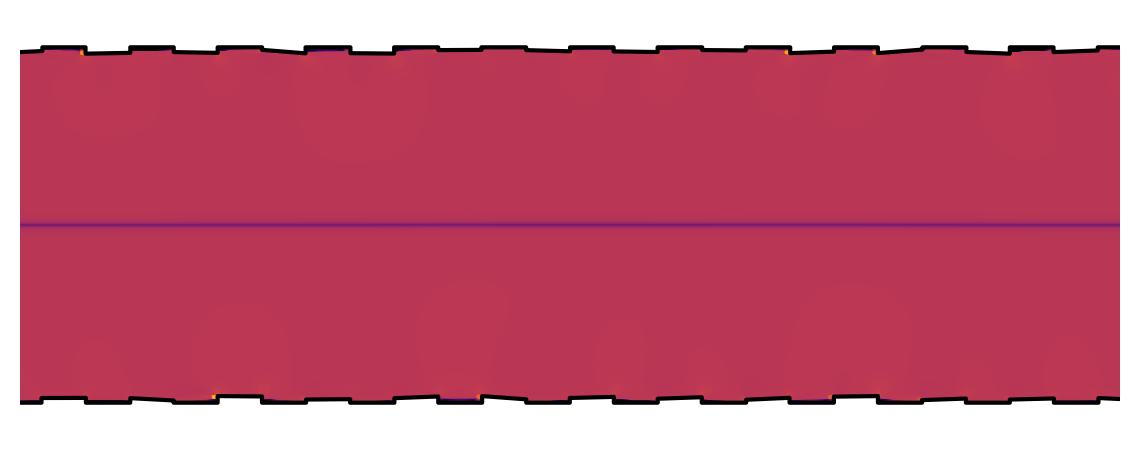}
        \hfill 
        \includegraphics[width=0.325\textwidth, trim=0mm 5mm 0mm 5mm, clip]{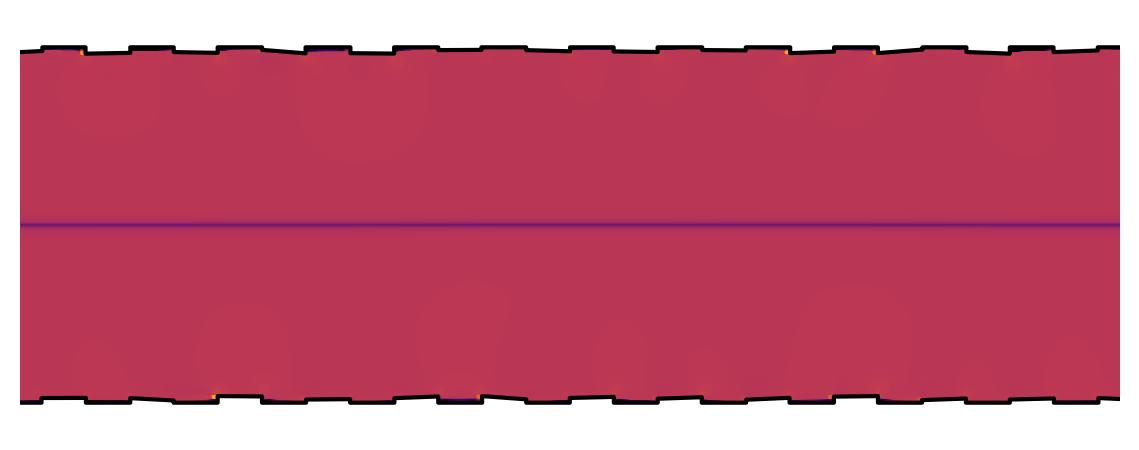}
        \subcaption{$(1-\beta) / \Reynolds = 10^{-3}$}
        \medskip
    \end{subfigure}
    \begin{subfigure}[b]{\textwidth}
        \centering
        \includegraphics[width=0.325\textwidth, trim=0mm 5mm 0mm 5mm, clip]{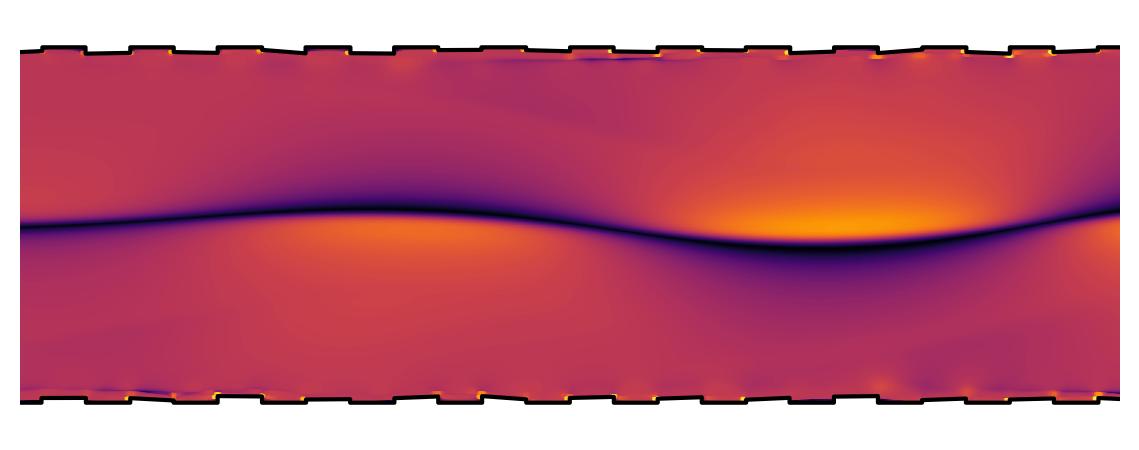}
        \hfill 
        \includegraphics[width=0.325\textwidth, trim=0mm 5mm 0mm 5mm, clip]{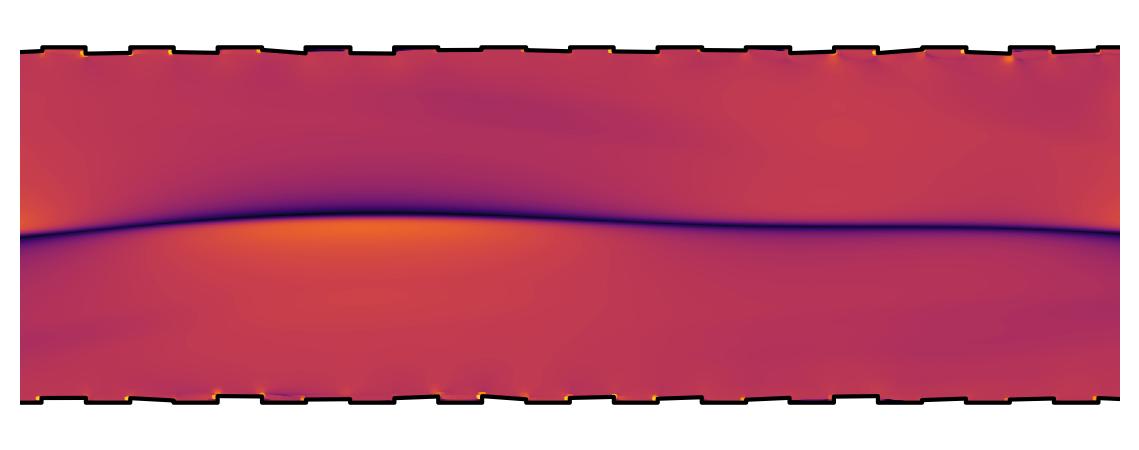}
        \hfill 
        \includegraphics[width=0.325\textwidth, trim=0mm 5mm 0mm 5mm, clip]{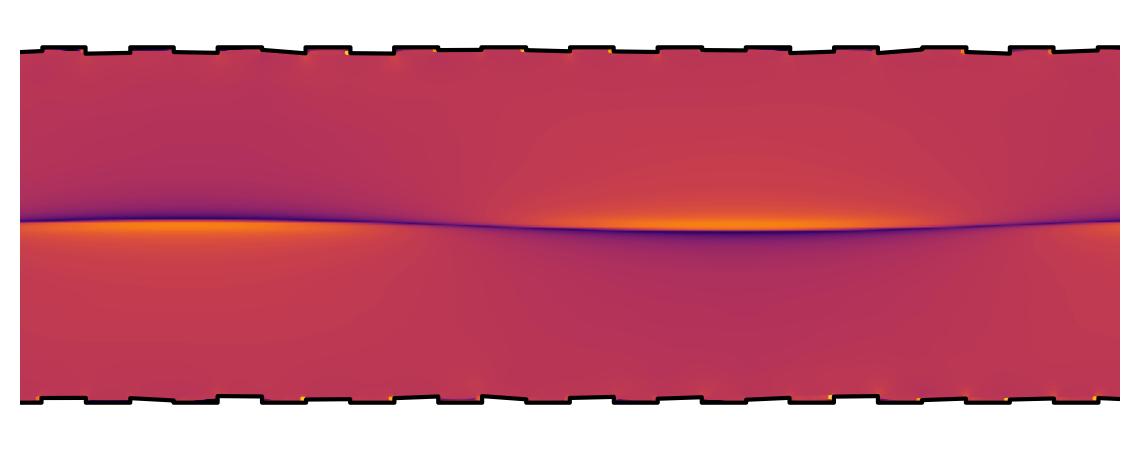}
        \subcaption{$(1-\beta) / \Reynolds = 10^{-4}$}
        \medskip
    \end{subfigure} 
    \begin{subfigure}[b]{\textwidth}
        \centering
        \includegraphics[width=0.325\textwidth, trim=0mm 5mm 0mm 5mm, clip]{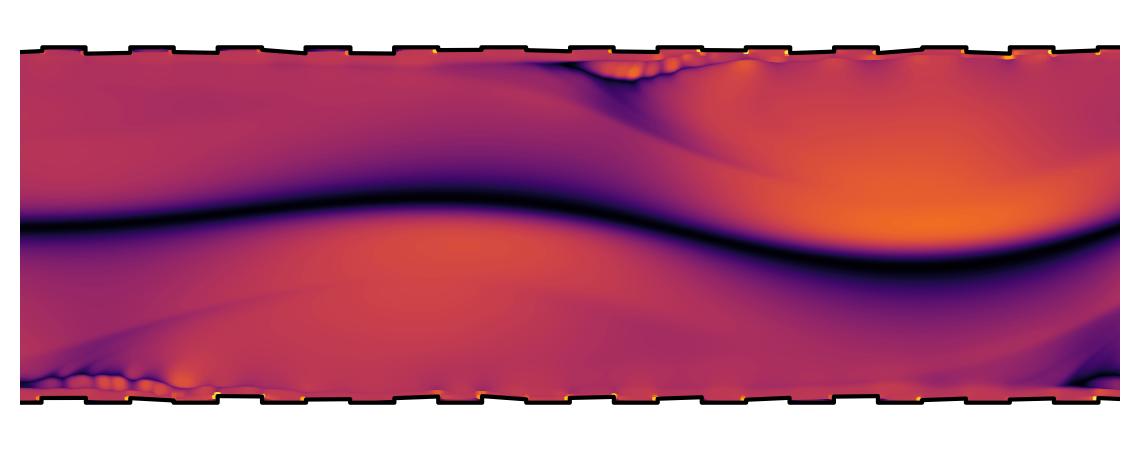}
        \hfill 
        \includegraphics[width=0.325\textwidth, trim=0mm 5mm 0mm 5mm, clip]{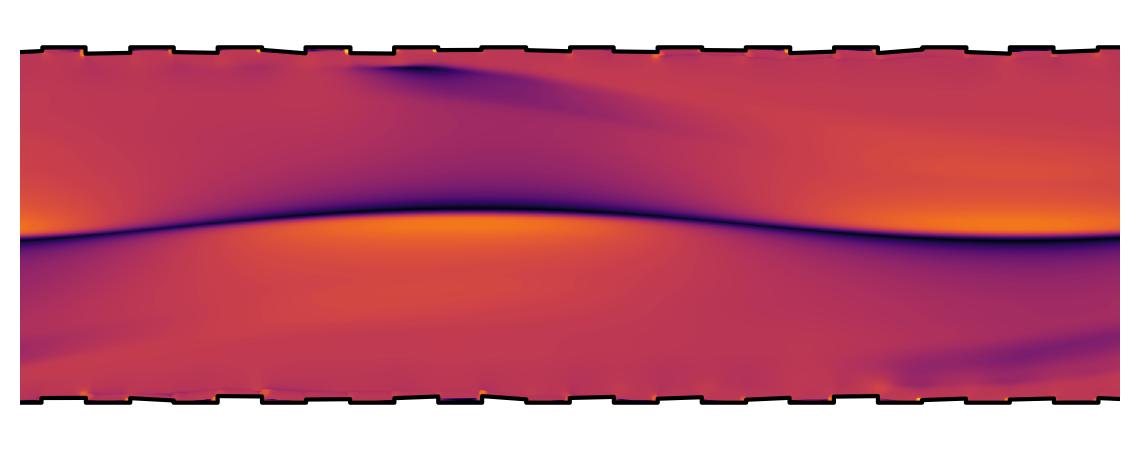}
        \hfill 
        \includegraphics[width=0.325\textwidth, trim=0mm 5mm 0mm 5mm, clip]{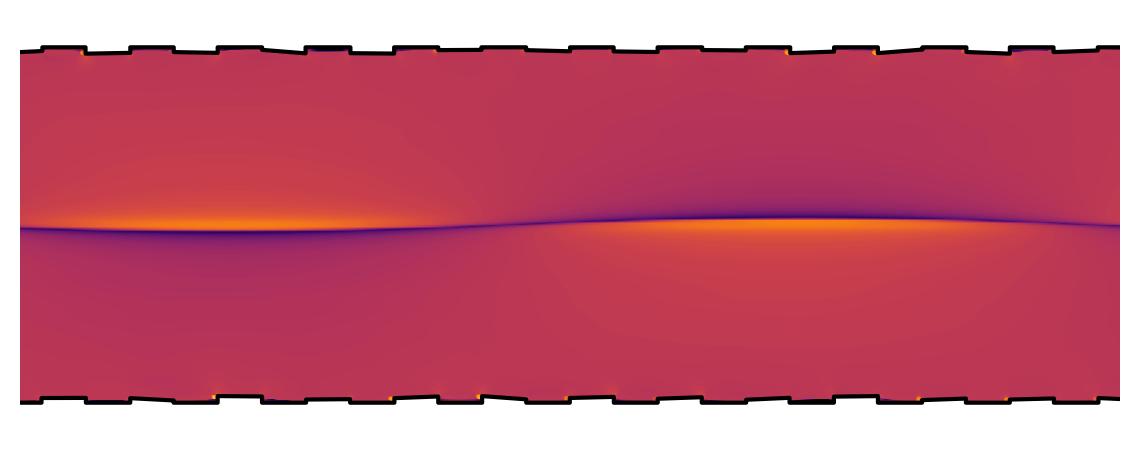}
        \subcaption{$(1-\beta) / \Reynolds = 10^{-5}$}
        \medskip
    \end{subfigure} 
    \begin{subfigure}[b]{\textwidth}
        \centering
        \includegraphics[width=0.325\textwidth, trim=0mm 5mm 0mm 5mm, clip]{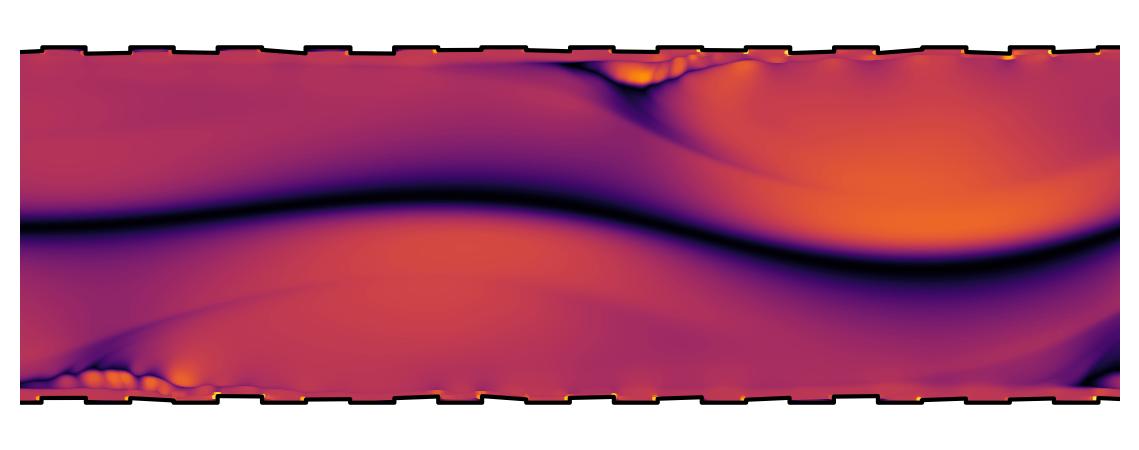}
        \hfill 
        \includegraphics[width=0.325\textwidth, trim=0mm 5mm 0mm 5mm, clip]{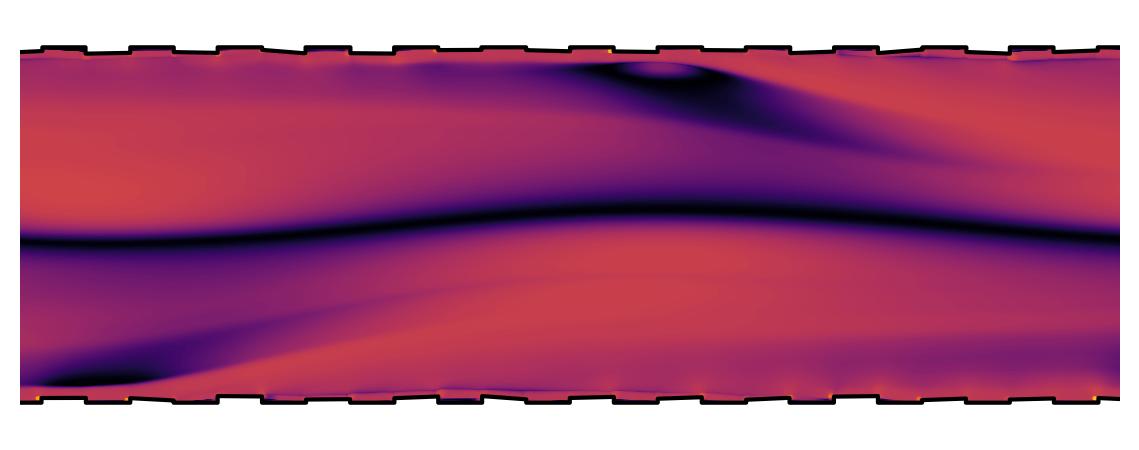}
        \hfill 
        \includegraphics[width=0.325\textwidth, trim=0mm 5mm 0mm 5mm, clip]{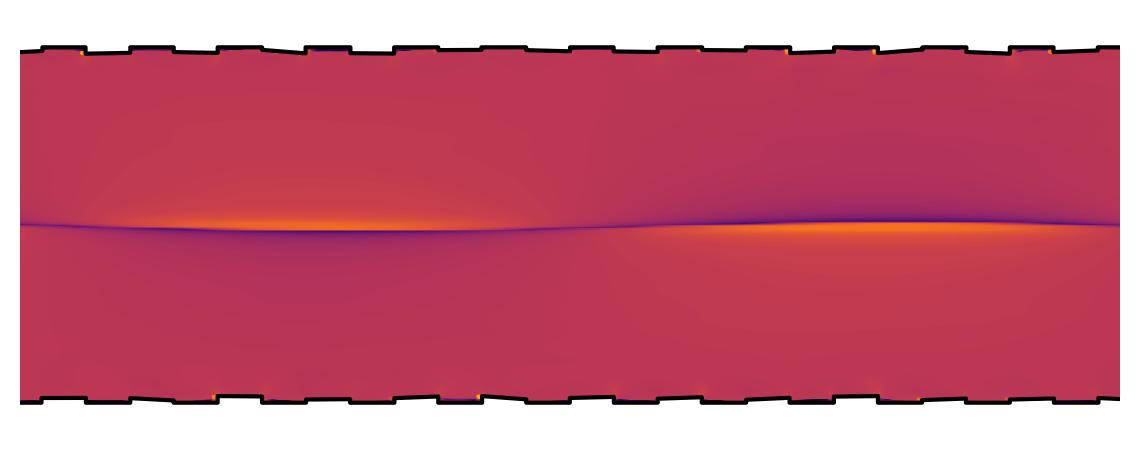}
        \subcaption{$(1-\beta) / \Reynolds = 10^{-6}$}
    \end{subfigure} 
    \caption{Visualization of $\omega$, for $\epsilon=10^{-2}$, $\Deborah = 1$, and various $(1-\beta) / \Reynolds$, at times $t=9,12,15$.}
    \label{fig:Omega_Method_various_gamma}
\end{figure}


\begin{figure}
    \centering
    \begin{subfigure}[b]{0.24\textwidth} \centering $t = 9$ \end{subfigure} \hfill 
    \begin{subfigure}[b]{0.24\textwidth} \centering $t = 10$ \end{subfigure} \hfill 
    \begin{subfigure}[b]{0.24\textwidth} \centering $t = 11$ \end{subfigure} \hfill 
    \begin{subfigure}[b]{0.24\textwidth} \centering $t = 12$ \end{subfigure}  
    \medskip

    \begin{subfigure}[b]{\textwidth}
        \centering
        \includegraphics[width=0.24\textwidth, trim=0mm 5mm 0mm 5mm, clip]{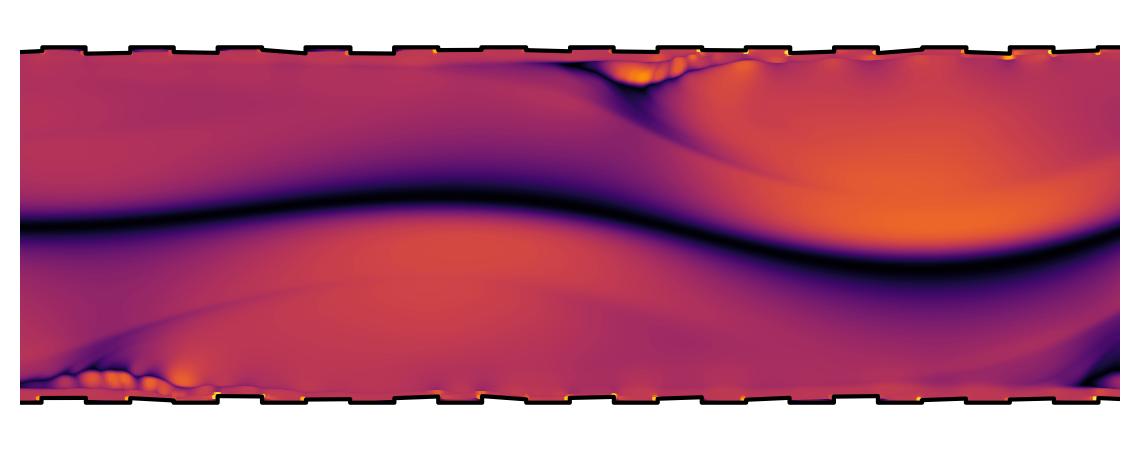}
        \hfill 
        \includegraphics[width=0.24\textwidth, trim=0mm 5mm 0mm 5mm, clip]{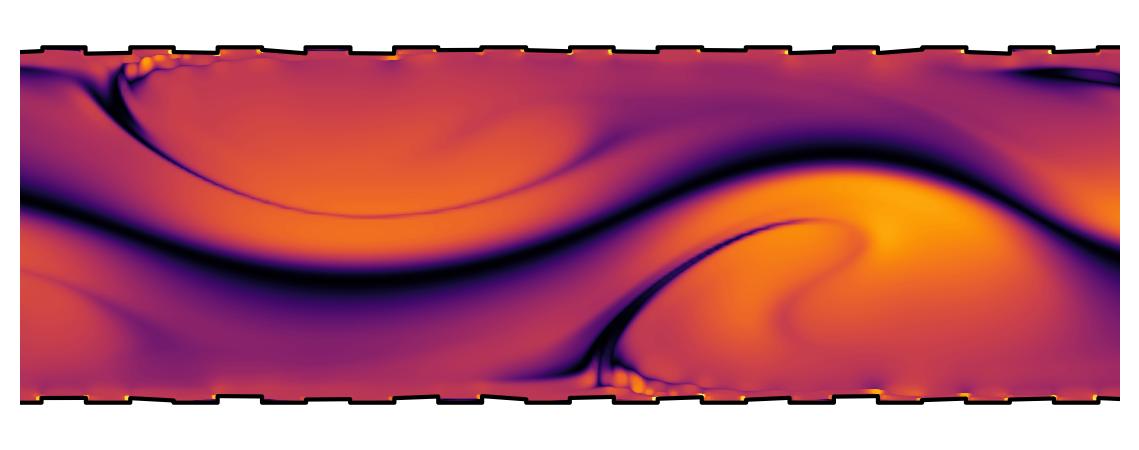}
        \hfill 
        \includegraphics[width=0.24\textwidth, trim=0mm 5mm 0mm 5mm, clip]{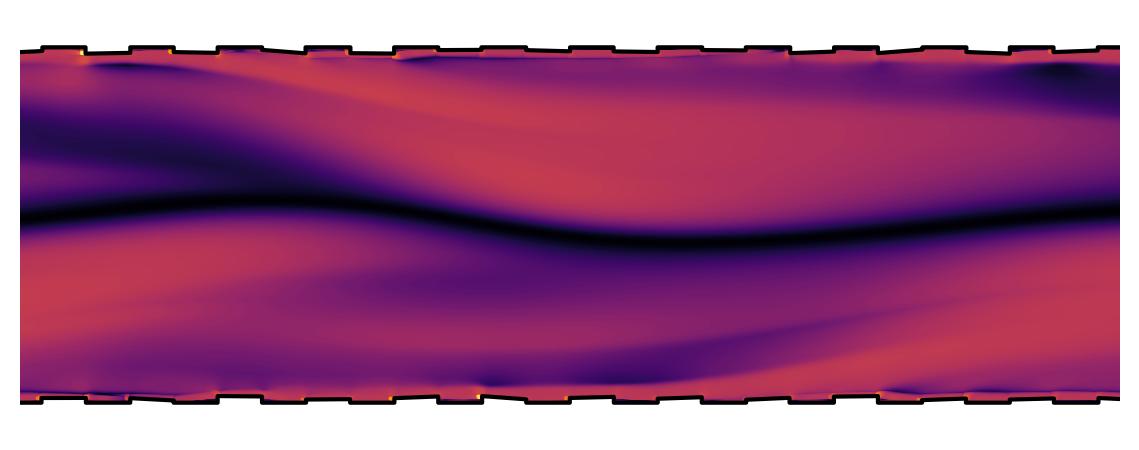}
        \hfill 
        \includegraphics[width=0.24\textwidth, trim=0mm 5mm 0mm 5mm, clip]{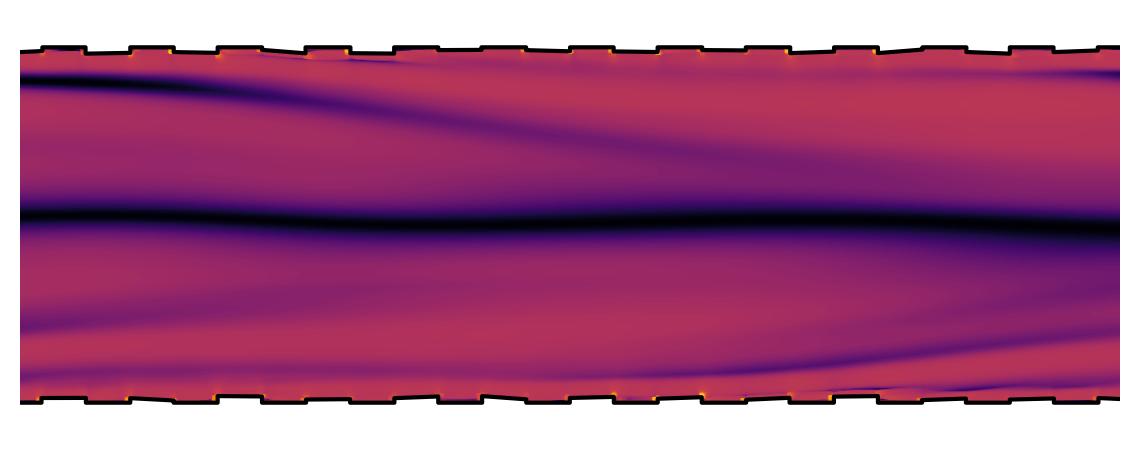}
        \subcaption{$\Deborah = 100$}
        \medskip
    \end{subfigure}
    \begin{subfigure}[b]{\textwidth}
        \centering
        \includegraphics[width=0.24\textwidth, trim=0mm 5mm 0mm 5mm, clip]{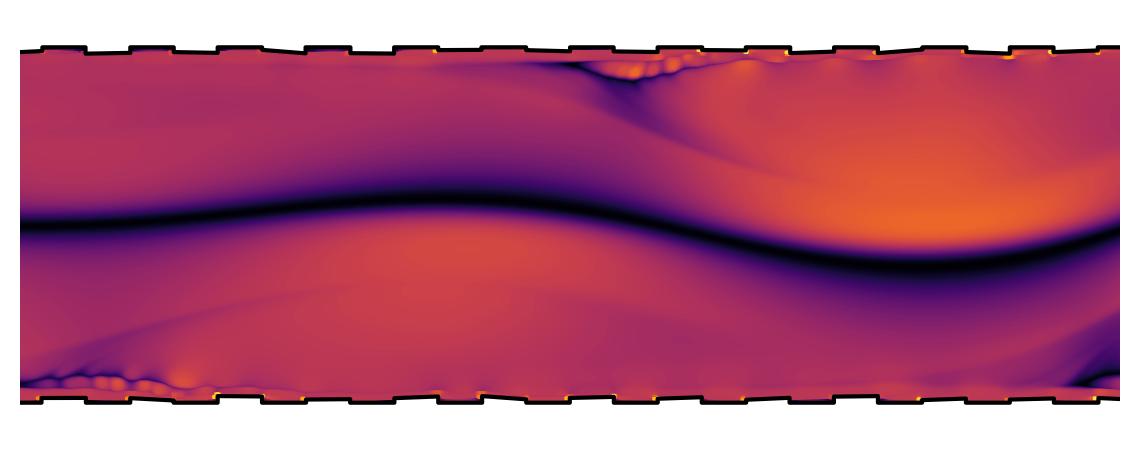}
        \hfill 
        \includegraphics[width=0.24\textwidth, trim=0mm 5mm 0mm 5mm, clip]{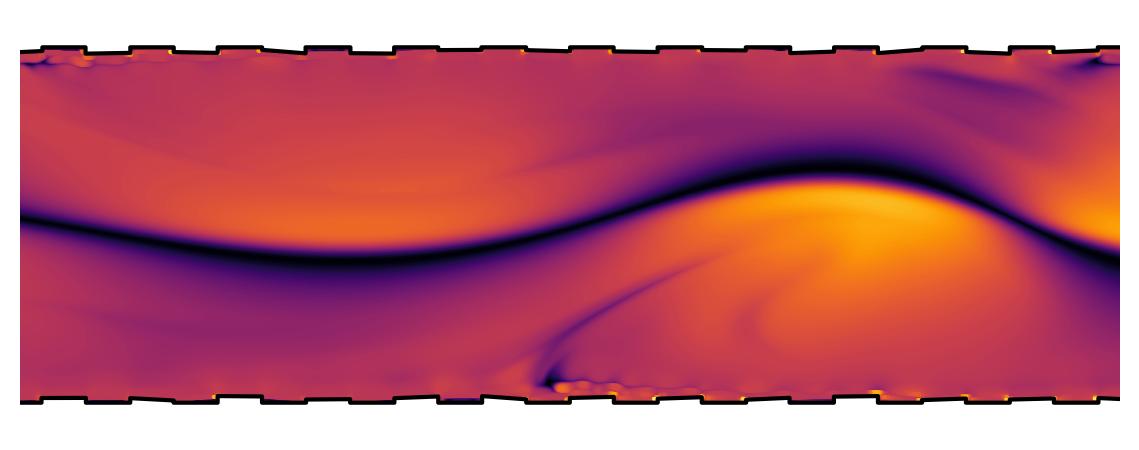}
        \hfill 
        \includegraphics[width=0.24\textwidth, trim=0mm 5mm 0mm 5mm, clip]{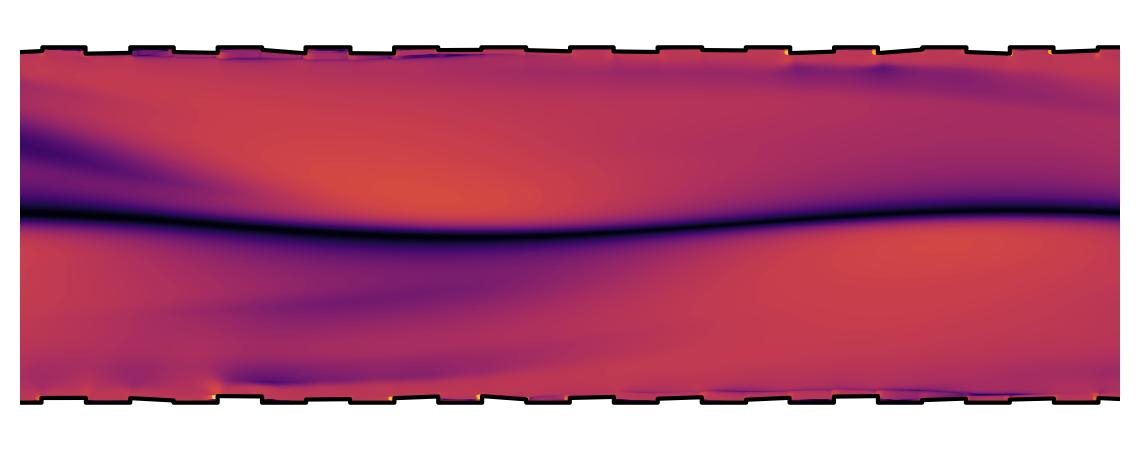}
        \hfill 
        \includegraphics[width=0.24\textwidth, trim=0mm 5mm 0mm 5mm, clip]{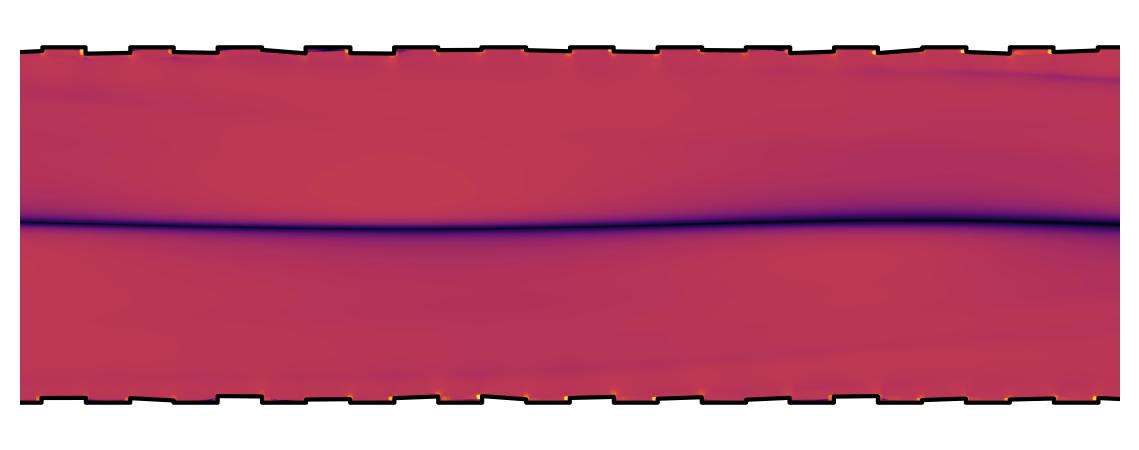}
        \subcaption{$\Deborah = 10$}
        \medskip
    \end{subfigure} 
    \begin{subfigure}[b]{\textwidth}
        \centering
        \includegraphics[width=0.24\textwidth, trim=0mm 5mm 0mm 5mm, clip]{Images/Case5_RWC_2D/half_channel_omega/Omega_half_t_9_NSFP__THP_dt_0.000100_kv_0.000030_g_0.000010_De_0.500000_comdiff_0.010000.jpg}
        \hfill 
        \includegraphics[width=0.24\textwidth, trim=0mm 5mm 0mm 5mm, clip]{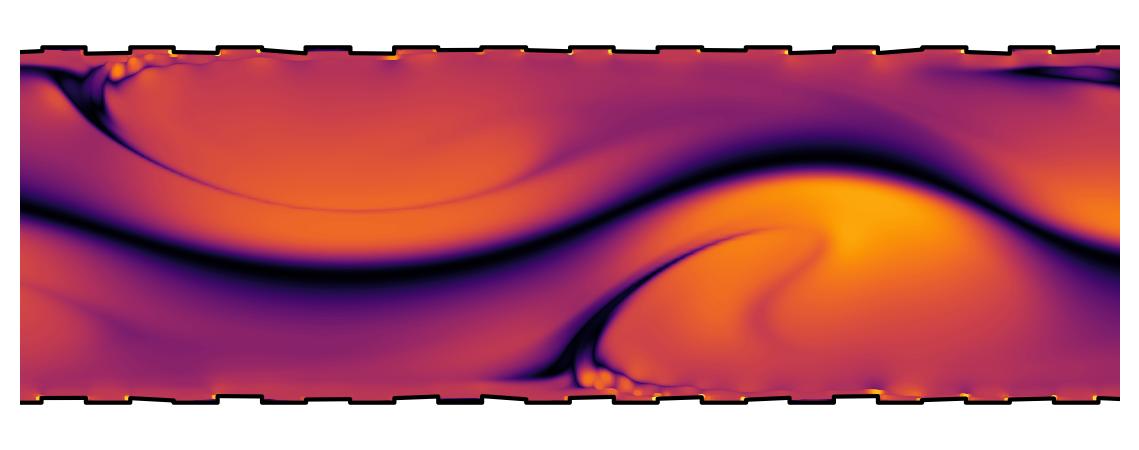}
        \hfill 
        \includegraphics[width=0.24\textwidth, trim=0mm 5mm 0mm 5mm, clip]{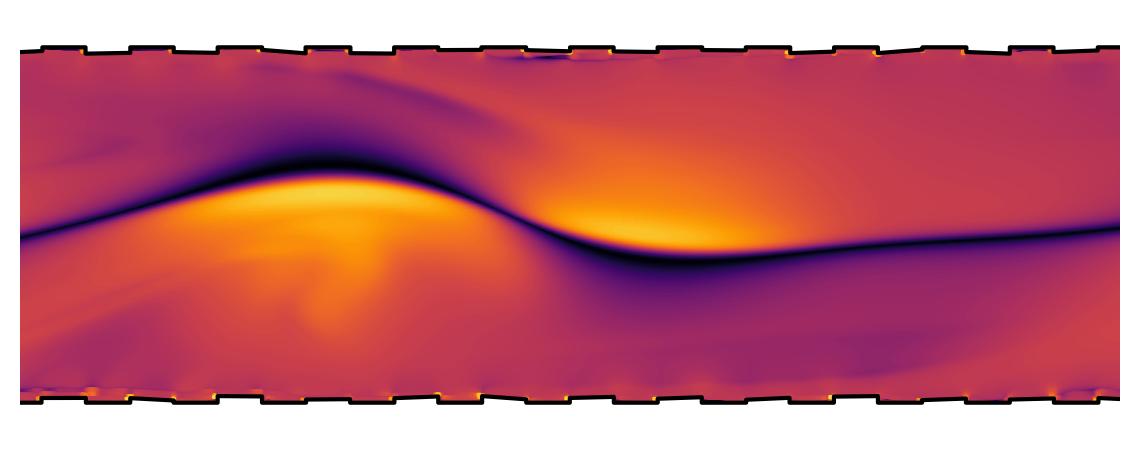}
        \hfill 
        \includegraphics[width=0.24\textwidth, trim=0mm 5mm 0mm 5mm, clip]{Images/Case5_RWC_2D/half_channel_omega/Omega_half_t_12_NSFP__THP_dt_0.000100_kv_0.000030_g_0.000010_De_0.500000_comdiff_0.010000.jpg}
        \subcaption{$\Deborah = 1$}
    \end{subfigure} 
    \caption{Visualization of $\omega$, for $\epsilon=10^{-2}$, $(1-\beta) / \Reynolds = 10^{-5}$, and various $\Deborah$, at times $t=9,10,11,12$.}
    \label{fig:Omega_Method_various_Deborah}
\end{figure}

Figure \ref{fig:Omega_Method_various_gamma} displays $\omega$ for $\Deborah=1$, $\comdiff=0.01$, $(1-\beta) / \Reynolds = 10^{-3},10^{-4},10^{-5},10^{-6}$ at $t=9,12,15$. The larger the viscosity contribution of the polymer molecules, the more stable the flow profile. For $(1-\beta) / \Reynolds = 10^{-3}$, the flow remains laminar throughout the simulation, while for smaller values, the flow develops turbulence before reverting to a laminar flow. The lower $(1-\beta) / \Reynolds$, the more pronounced the turbulence and the longer it takes until the flow stabilizes. 

We show the influence of the Deborah number in Figure \ref{fig:Omega_Method_various_Deborah}. Therefore, we set $\comdiff=0.01$, $(1-\beta) / \Reynolds = 10^{-5}$, and consider $\Deborah=1, 10, 100$.  Independently of the value of $\Deborah$, the simulations develop turbulence at $t=9$, with all simulations showing similar turbulence characteristics. At $t=10$, the results for $\Deborah=1$ and $\Deborah=100$ look similar, while the profile for $\Deborah=10$ appears less turbulent. At $t=11$ and $t=12$, in all simulations, the turbulence is damped, and the flow returns to a more laminar state, whereby in the case of $\Deborah=10$, the turbulence is damped the fastest. The simulation results agree with the onset of drag-reducing effects at moderate Deborah numbers starting at $\Deborah \approx 10$ \cite{graham2014drag}. 

\subsection{Pipeline strain relief}

We show the applicability of the numerical scheme in three dimensions, considering the flow through a geometry modeled after a pipeline strain relief consisting of a circular pipe of radius $0.5$ with $4$ right-angle turns, such that $\Omega \subset [0,12] \times [0,4] \times [0,1]$. The mesh comprises $125\,954$ hexahedral elements, resulting in a total number of $12\,335\,540$ degrees of freedom for the coupled MNSFP system. We use a time step of $\Delta t = 5 \cdot 10^{-4}$. 
For the MFP equation, we consider homogeneous Neumann boundary conditions and, for the velocity, no-slip boundary conditions at the pipeline's walls, homogeneous Neumann boundary conditions at the outlet of the domain, and a quadratic inflow profile at the inlet centered at $\Coordinate = (0,\, 0.5,\, 0.5)\Transpose$
\begin{equation}
    \Velocity_\text{in}(\Time, \Coordinate) 
    = \sin^2\!\left(\frac{\pi}{2} \min\{t, 1\}\right) 
    \begin{pmatrix} 1 - 4\big((x_2-0.5)^2 + (x_3-0.5)^2\big), & 0, & 0\end{pmatrix}\Transpose.
\end{equation}
We compare the solution of the pure solvent fluid with the solution of the NSFP, for $\epsilon = 1$, $\Deborah = 1$, $\beta / \Reynolds = 0.001$, and $(1 - \beta) / \Reynolds = 20$.
The pure solvent flow results in turbulent zones, especially after the last $90^\circ$ bend of the geometry, but no fully turbulent flow, while the dilute polymeric fluid shows no onset of turbulence. 
The streamlines of the two simulations at $t=10$ are displayed in Figure \ref{fig:psr}.

\begin{figure}
    \centering
    \begin{subfigure}[b]{0.49 \textwidth}
        \centering
        \includegraphics[width=\textwidth, trim={5.5cm 1.5cm 2cm 0.75cm},clip]{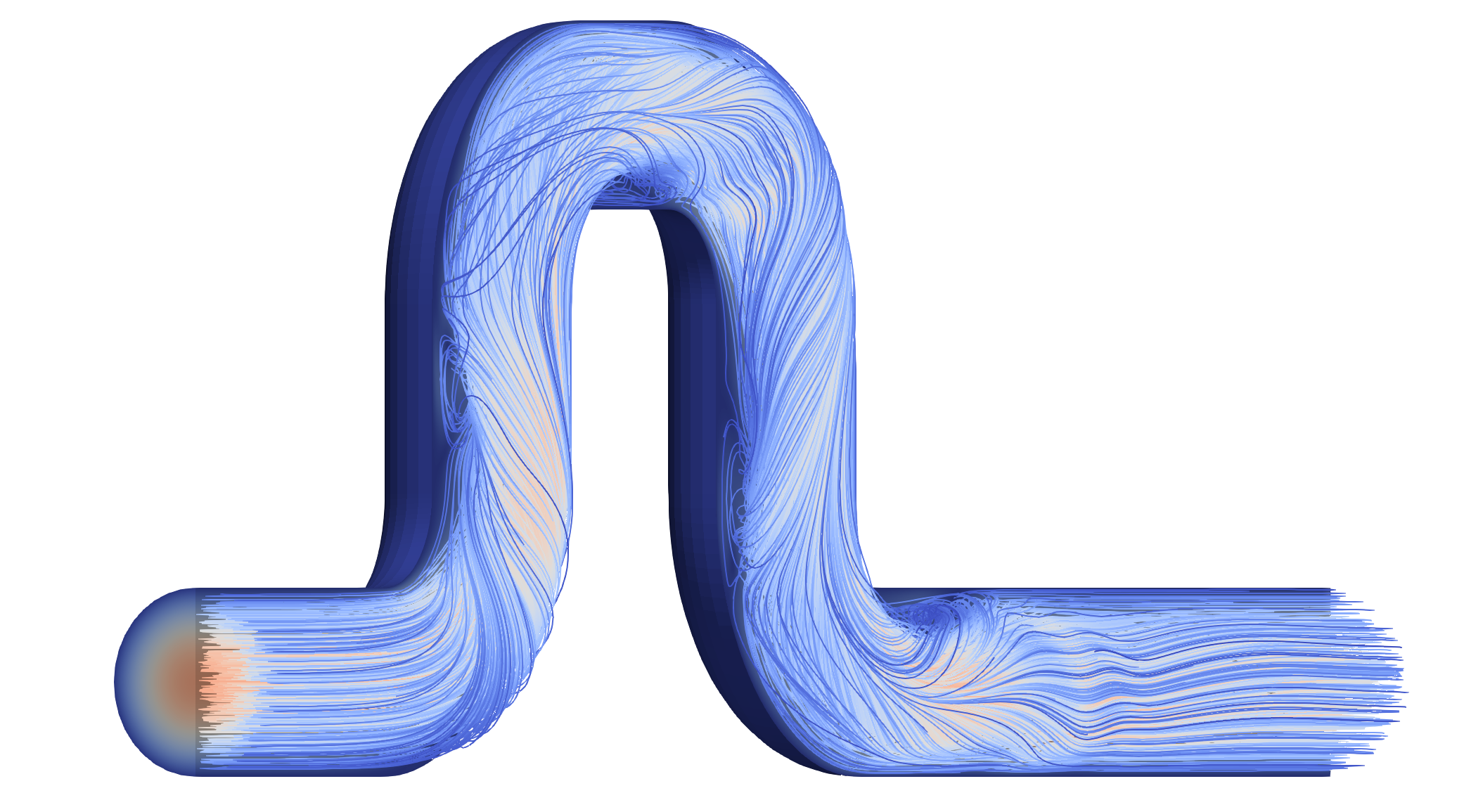}
        \caption{Pure solvent fluid}
        \label{fig:psr_NS}
    \end{subfigure}
    \begin{subfigure}[b]{0.49\textwidth}
        \centering
        \includegraphics[width=\textwidth, trim={5.5cm 1.5cm 2cm 0.75cm},clip]{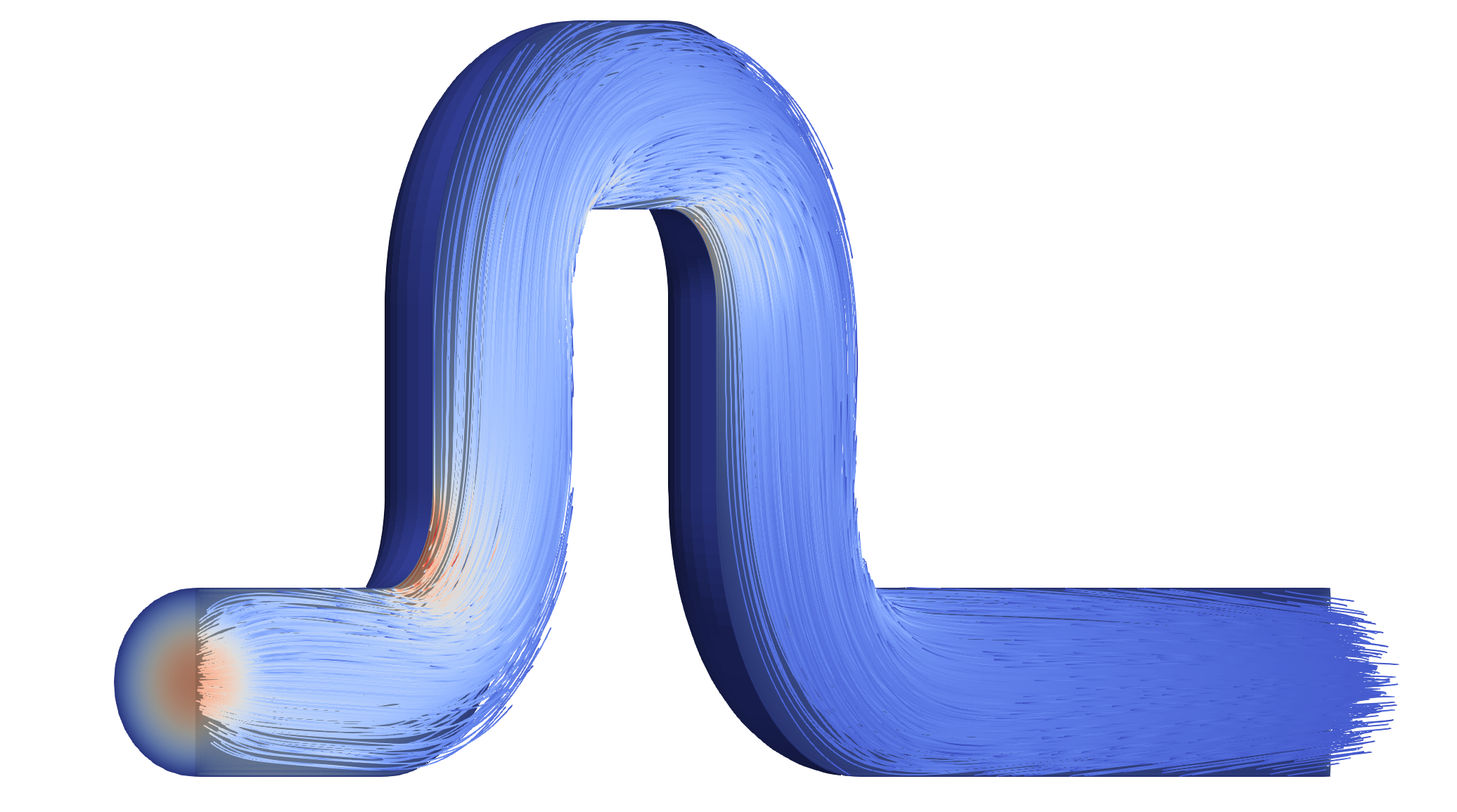}
        \caption{Dilute polymeric fluid}
        \label{fig:psr_NSFP}
    \end{subfigure}
    \caption{Comparison of streamlines of (a) the pure solvent and (b) the dilute polymeric fluid for the pipeline strain relief geometry at $t=10$. The coloring is according to the velocity magnitude $\Abs{\Velocity}$, which ranges from 0 to 1.41.}
    \label{fig:psr}
\end{figure}

\section{Conclusion}
\label{sec:conclusion}
We introduced a purely macroscopic model for the simulation of dilute polymeric fluids with memory effects, proposed a numerical scheme for discretization, and conducted numerical simulations of dilute polymeric fluids in the turbulent regime. 
Our model derives from the application of the Hermite spectral method to the time-fractional Hookean-type Navier--Stokes--Fokker--Planck system, the coupling structure of the Hermite spectral modes, and the orthogonality of Hermite functions in combination with the polynomiality of Kramer's expression. 
We prove that by selecting the Hermite scaling parameter $\scaling = 1/\sqrt{2}$, the coupling tensor exactly resembles the coupling in the micro-macro model, providing a more precise and consistent representation than the arbitrary choices in previous works. 
Our numerical scheme applies a kernel compression method to the time-fractional derivative and a projection method to the Navier--Stokes equations. Combining second-order time integration with extrapolation of the coupling terms, we achieve linear convergence of the fully coupled system independent of the order of the time-fractional derivative.
The highly efficient implementation allows us to study the influence of polymer molecules with memory effects in the turbulent regime, revealing that memory effects decrease the polymer-induced force and, thus, the drag-reducing effect in dilute polymeric fluids.  
The absence of a macroscopic closure of the FENE-type Navier--Stokes--Fokker--Planck system poses a significant challenge in simulating time-fractional dilute polymeric fluids, which can not be modeled using linear spring potentials.
Therefore, future research should prioritize addressing this gap, starting with macroscopic models such as the time-fractional FENE-P model, with the ultimate goal being the simulation of a fully resolved time-fractional FENE-type Navier--Stokes--Fokker--Planck system.

\appendix

\section{Proof of Lemma \ref{lem:tau_equals_tau2}}
\label{apdx:proof_lemma}
\begin{proof}
    Note that for $f \in L^2_{1/w}(\R)$, $f/w \in L^2_{w}(\R)$, and thus, 
    \begin{equation}
        (f- \Tilde{\Pi}_N f, g)_{1/w} = 0, \quad \forall g \in \text{span}\{\HermiteFunction{k}\}_{k=0}^N, 
    \end{equation}
    cf. (7.125) and (7.129) in \cite{shen2011spectral}, and we obtain the best $L^2_{1/w}(\R)$ approximation of $\FPpdf$ using Hermite functions with scaling parameter $\scaling = 1 / \sqrt{2}$ up to degree $N$ as in \eqref{eq:f2}. It remains to show \eqref{eq:tau_equals_tauN}. 
    As outlined in Section \ref{subsec:mm_extra_stress}, we exploit that for $ N \geq 2$ 
    \begin{equation}
         \int_\ConfigurationSpace \left(\Configuration \Configuration \Transpose - \Id \right) f_N(\Configuration) \,\Diff \Configuration 
        = \int_\ConfigurationSpace \left(\Configuration \Configuration \Transpose - \Id \right) f_2 (\Configuration) \,\Diff \Configuration, 
    \end{equation}
    and thus, \eqref{eq:tau_equals_tauN} becomes  
    \begin{equation}
        \label{eq:tau_equals_tau2}
        \ExtraStress := \int_\ConfigurationSpace \left(\Configuration \Configuration \Transpose - \Id \right) f(\Configuration) \,\Diff \Configuration 
        = \int_\ConfigurationSpace \left(\Configuration \Configuration \Transpose - \Id \right) f_2 (\Configuration) \,\Diff \Configuration =: \ExtraStress_2, 
    \end{equation}
    which we do by explicitly calculating $\ExtraStress$ and $\ExtraStress_2$. 
    Due to symmetry, it is sufficient to consider one diagonal entry and one off-diagonal each. 
    For the calculation, we repeatedly apply the integral identities  
    \begin{alignat}{3}
        \label{eq:apdx_integral_constant}
        \int_\R \exp{(-\xi x^2 - \mu x)} \,\Diff x 
        & = \sqrt{\pi} \frac{1 }{\sqrt{\xi}} \exp{\left(\frac{\mu ^2}{4\xi} \right)}, \quad 
        && \text{ for } \xi > 0, \mu \in \R, \\ 
        \label{eq:apdx_integral_linear}
        \int_\R x \exp{(-\xi x^2 - \mu x)} \,\Diff x 
        & = \sqrt{\pi} \frac{\mu }{2 \xi^{3/2}} \exp{\left(\frac{\mu ^2}{4\xi} \right)}, \quad 
        && \text{ for } \xi > 0, \mu \in \R, \\ 
        \label{eq:apdx_integral_quadratic}
        \int_\R x^2 \exp{(-\xi x^2 - \mu x)} \,\Diff x 
        & = \sqrt{\pi}\frac{(2\xi+\mu ^2) }{4 \xi^{5/2}}\exp{\left(\frac{\mu ^2}{4\xi} \right)}, \quad 
        && \text{ for } \xi > 0, \mu \in \R .
    \end{alignat}
    For better readability, we identify the entries of the matrix $\tensor{C}$ as
    \begin{equation}
        \tensor{C} 
        = \left(\begin{array}{ccc}
             A & D & E \\
             D & B & F \\ 
             E & F & C 
        \end{array}\right).
    \end{equation}
    For $\ExtraStress$, we obtain with $f(\Configuration) = \exp(-\Configuration\Transpose\tensor{C}\Configuration)$
    \begin{align}
        \ExtraStress_{11} 
        & = \int_{\ConfigurationSpace}  
        \left({q_1}^2-1\right) f(\Configuration) \Diff \Configuration 
        \\ \label{eq:tau_11}
        & = \frac{\pi^{3/2}}{A^{5/2}} \left(
        \frac{1}{\theta^{1/2}} 
        \left(\frac{A/2-A^2}{\sigma^{1/2}} 
        + \frac{ D^2}{2 \sigma^{3/2}} \right)
        + \frac{1}{2\theta^{3/2}} 
        \left(\frac{E^2}{\sigma^{1/2}} 
        + \frac{ 2 DE \zeta}{\sigma^{3/2}}
        + \frac{D^2\zeta^2}{\sigma^{5/2}}\right)
        \right) 
    \end{align}
    and 
    \begin{align}
        \label{eq:tau_12}
        \ExtraStress_{12} 
        = \int_{\ConfigurationSpace}  
        q_1 q_2 f(\Configuration)
        \Diff \Configuration
        = \frac{\pi^{3/2}}{2A^{3/2}} 
        \left(
        \frac{D}{\theta^{1/2}\sigma^{3/2}} 
        + \frac{1}{\theta^{3/2}} \left(
            \frac{E\zeta}{\sigma^{3/2}} 
          + \frac{D\zeta^2}{\sigma^{5/2}}
        \right)
        \right), 
    \end{align}
    whereby we introduced 
    \begin{align}
        \label{eq:sigma_zeta_theta}
        \sigma := B - \frac{D^2}{A}, \quad
        \zeta  := F - \frac{DE}{A}, \quad 
        \theta := C - \frac{E^2}{A} - \frac{\zeta^2}{\sigma}.   
    \end{align} 
    During the calculation of $\ExtraStress_{11}$ and $\ExtraStress_{12}$, we require $A>0$, $\sigma>0$, and $\theta>0$, to fulfill $\xi > 0$ in equations \eqref{eq:apdx_integral_constant}-\eqref{eq:apdx_integral_quadratic}. The inequalities are exactly Sylvester's criterion applied to $\tensor{C}$, which is positive definite. 
    
    For the Hermite approximation, we note that, for $a=1/\sqrt{2}$, $\chi_0 = 0$, and the entries reduce to  
    \begin{equation}
        \left(\ExtraStress_2\right)_{11} = \chi_2 \HermiteModes{2,0,0}, \quad 
        \left(\ExtraStress_2\right)_{12} = \chi_1 \HermiteModes{1,1,0}.
    \end{equation}
    Again, applying \eqref{eq:apdx_integral_constant}-\eqref{eq:apdx_integral_quadratic} with the same choices for $\xi$ as before, we obtain 
    \begin{align}
        \left(\ExtraStress_2\right)_{11} 
        & = \chi_2 \phi_{2,0,0} 
        = \chi_2 \frac{1}{\sqrt{8}} \left(\frac{\sqrt{a}}{\sqrt[4]{\pi}}\right)^{3} 
        \int_{\ConfigurationSpace}  
        f(\Configuration) (4a^2 {q_1}^2 - 2)
        \, \Diff \Configuration \\ 
        & = 
        \chi_2 \frac{1}{\sqrt{8}} \left(\frac{\sqrt{a}}{\sqrt[4]{\pi}}\right)^{3} 
        \frac{\pi^{3/2}}{A^{5/2}}
        \Bigg[ \frac{2}{\theta^{1/2}}
            \Bigg(
                \frac{A a^2 - A^2}{\sigma^{1/2}} + \frac{D^2 a^2}{\sigma^{3/2}} 
            \Bigg) + \frac{a^2}{2 \theta^{3/2}} \Bigg( 
                \frac{4E^2}{\sigma^{1/2}}
                + \frac{8 D E \zeta}{\sigma^{3/2}}
                + \frac{4 D^2 \zeta^2 }{\sigma^{5/2}}
            \Bigg) 
        \Bigg]\\ \label{eq:tau_2_11}
        & = \frac{1}{2a^2} \frac{\pi^{3/4}}{A^{5/2}}  
        \Bigg[ \frac{1}{\theta^{1/2}}
            \Bigg(
                \frac{A a^2 - A^2}{\sigma^{1/2}} + \frac{D^2 a^2}{\sigma^{3/2}} 
            \Bigg) + \frac{a^2}{\theta^{3/2}}\Bigg( 
                \frac{E^2}{\sigma^{1/2}}
                + \frac{2 D E \zeta}{\sigma^{3/2}}
                + \frac{D^2 \zeta^2 }{\sigma^{5/2}}
            \Bigg) 
        \Bigg], 
    \end{align}
    where we reused the coefficients introduced in \eqref{eq:sigma_zeta_theta}.    
    Comparing, \eqref{eq:tau_11} and \eqref{eq:tau_2_11}, we observe that, for $a = 1 / \sqrt{2}$, $\ExtraStress_{11}$ and $\left(\ExtraStress_2\right)_{11}$ coincide.
    We proceed analogously for the off-diagonal entry and compute   
    \begin{align}
        \left(\ExtraStress_2\right)_{12} 
        & = \chi_1 \phi_{1,1,0} 
        = \chi_1 \frac{1}{2} \left(\frac{\sqrt{a}}{\sqrt[4]{\pi}}\right)^{3} 
        \int_{\ConfigurationSpace}
        f(\Configuration) 4 a^2 q_1 q_2
        \, \Diff \Configuration \\ 
        & = \chi_1 \frac{1}{2} \left(\frac{\sqrt{a}}{\sqrt[4]{\pi}}\right)^{3} 
        4a^2 \frac{\pi^{3/2}}{2A^{3/2}}
        \Bigg[
            \frac{D}{\theta^{1/2} \sigma^{3/2}}  
            + \frac{1}{\theta^{3/2}} \left(\frac{E \zeta}{\sigma^{3/2}} + \frac{D\zeta^2}{\sigma^{5/2}} \right) 
        \Bigg], \\ \label{eq:tau_2_12}
        & = \frac{\pi^{3/2}}{2A^{3/2}}
        \Bigg[
            \frac{D}{\theta^{1/2} \sigma^{3/2}}  
            + \frac{1}{\theta^{3/2}} \left(\frac{E \zeta}{\sigma^{3/2}} + \frac{D\zeta^2}{\sigma^{5/2}} \right) 
        \Bigg]. 
    \end{align}           
    Therefore, $\left(\ExtraStress\right)_{12} = \left(\ExtraStress_2\right)_{12}$ independently of the choice $a$.
\end{proof}

\data{The source code will be available on GitHub after acceptance.}

\funding{The work of BW was supported by the Federal Ministry of Education
and Research (BMBF) as part of the "Multi-physics simulations for Geodynamics on heterogeneous
Exascale Systems" (CoMPS) project (FKZ 16ME0651) inside the federal research program on 
"High-Performance and Supercomputing for the Digital Age 2021-2024 -- Research and Investments in
High-Performance Computing."}

\bibliographystyle{elsarticle-num}
\bibliography{references}

\end{document}